\documentclass[11pt]{amsart}
\usepackage[dvipsnames,usenames]{color}
\usepackage{hyperref}
\usepackage{graphicx}
\usepackage{epsfig}
\usepackage[latin1]{inputenc}
\usepackage{amsmath}
\usepackage{amsfonts}
\usepackage{amssymb}
\usepackage{amsthm}
\usepackage{amscd}
\usepackage{verbatim}
\usepackage{subfigure}
\usepackage{pinlabel}
\usepackage{stmaryrd}
\usepackage{enumerate, enumitem}
\usepackage{floatrow}
\usepackage{graphicx, array, blindtext}
\usepackage{float}
\usepackage{pifont}

\usepackage[textsize=scriptsize]{todonotes}

\usepackage{tikz}
\usetikzlibrary{cd}
\usetikzlibrary{arrows}
\usetikzlibrary{decorations.pathreplacing}
\usetikzlibrary{positioning,shapes}
\usepackage{caption}
\usepackage{verbatim}

\newtheorem{theorem}{Theorem}[section]
\newtheorem{Theorem}{Theorem}
\newtheorem{lemma}[theorem]{Lemma}
\newtheorem{proposition}[theorem]{Proposition}

\newtheorem{Conjecture}[Theorem]{Conjecture}

\theoremstyle{definition}
\newtheorem{definition}[theorem]{Definition}

\newtheorem{question}{Question}[section]
\theoremstyle{remark}

\def\F{\mathbb{F}}

\def\Z{\mathbb{Z}}

\def\bbT{\mathbb{T}}

\def\cC{\mathcal{C}}

\def\cC{\mathcal{C}}

\def\cH{\mathcal{H}}

\def\cR{\mathcal R}

\def\cU{\mathcal{U}}
\def\cV{\mathcal{V}}

\usepackage{bbm}
\usepackage{graphicx}
\usepackage{yfonts}
\usepackage[sharp]{easylist}

\def\CFKUV{\CFK_\cR} 

\newcommand{\simeqd}{\mathrel{\rotatebox[origin=c]{-180}{$\simeq$}}}

\def\CFh{\widehat{\operatorname{CF}}}

\def\HFh{\widehat{\operatorname{HF}}}

\def\CFK{\operatorname{CFK}}
\def\CFKi{\operatorname{CFK}^{\infty}}

\def\CFKh{\widehat{\operatorname{CFK}}}

\def\HFKh{\widehat{\operatorname{HFK}}}
\def\CFAh{\widehat{\operatorname{CFA}}}
\def\CFDh{\widehat{\operatorname{CFD}}}
\def\CFDDh{\widehat{\operatorname{CFDD}}}
\def\CFDAh{\widehat{\operatorname{CFDA}}}
\def\CFAAh{\widehat{\operatorname{CFAA}}}

 \DeclareMathOperator{\gr}{gr}

\def\muv{\hspace{.25cm} \text{mod} \hspace{.15cm}(\cU, \cV)}

\definecolor{darkgreen}{rgb}{0,0.5,0}
\definecolor{purple}{rgb}{0.5,0,0.5}

\author[]{Sally Collins}{\thanks{The author was supported by NSF GRFP grant DGE-2039655, and by NSF grant DMS-1928930 while in residence at the Simons Laufer Mathematical Sciences Institute in Berkeley, California during the Fall 2022 semester}}
\address {School of Mathematics, Georgia Institute of Technology, Atlanta, GA 30332}
\email{sallycollins@gatech.edu}

\title{Homology cobordism, smooth concordance, and the figure eight knot}
\begin{document}
\maketitle

\begin{abstract}
The $0$-surgeries of two knots $K_1$ and $K_2$ are homology cobordant rel meridians if there exists a $\Z$-homology cobordism $X$ between them such that the two knot meridians are in the same homology class in $H_{1}(X,\mathbb{Z})$. In this paper, we give a pair of rationally slice knots which are not smoothly concordant but whose $0$-surgeries are homology cobordant rel meridians. One knot in the pair is the figure eight knot, which has concordance order two; all previous examples of such pairs of knots are infinite order.

\end{abstract}

\section{Introduction}
\label{sec:1intro}

Two oriented knots $K_1$ and $K_2$ in $S^3$ are {\em smoothly concordant}, denoted $K_1 \sim K_2$, if they cobound a smooth annulus in $S^3\times [0,1]$.  The set of isotopy classes of knots forms a monoid, and that monoid modulo this smooth concordance relation yields an abelian group with binary operation connected sum. We call this group the {\em smooth concordance group}, denoted $\cC$, whose identity element is the class of {\em smoothly slice knots}. There are several variations on {$\cC$}, including the {\em rational concordance group} {$\cC_{\mathbb{Q}}$}, where two knots are equivalent if they cobound a smooth annulus in a rational homology cobordism, and the {\em topological concordance group} {$\cC_{\text{TOP}}$}, where we relax the conditions on the annulus the knots cobound to only require that it be locally flat. 

We let $S^3_n(K)$ denote the 3-manifold that results from performing $n$-framed surgery on a knot $K$ in $S^3$. In general, two 3-manifolds $Y_1$ and $Y_2$ are {\em $\Z$-homology cobordant} if there exists a smooth, compact, oriented 4-manifold $X$ such that $\partial X = (-Y_1) \cup Y_2$ and $H_{*}(X, Y_i; \Z) = 0$ for $i= 1,2$. If we let $Y_1 = S^3_n(K_1)$ and $Y_2 = S^3_n(K_2)$ for some $K_1, K_2$ knots in  $S^3$ and some $n \in \Z$, then the $n$-surgeries of $K_1$ and $K_2$ are {\em $\mathbb{Z}$-homology cobordant rel meridians} if they are $\Z$-homology cobordant by cobordism $X$ and the positively oriented meridians of the $K_i$ are representatives of the same homology class in $H_1(X; \Z)$.

It is a natural question to ask whether there is a relationship between the homeomorphism type of $S^3_n(K)$ and the knot type of $K$; in the Kirby problem list \cite{Kirby}, Akbulut and Kirby conjectured in Problem 1.19 and subsequent remarks that two knots with the same $0$-surgeries will be concordant knots. This conjecture inspires the following question:

\begin{question}\label{bigQ}
If $S^3_0(K_1)$ and $S^3_0(K_2)$ are $\Z$-homology cobordant rel meridians, are $K_1$ and $K_2$ smoothly concordant?
\end{question}

In \cite{Cochran_2013}, Cochran-Franklin-Hedden-Horn give a negative answer to this question by providing a family of topologically slice knots which are not pairwise smoothly concordant, but whose $0$-surguries satisfy the hypothesis of \ref{bigQ}. All knots in this family have infinite order in $\mathcal{C}$, and are distinguishable as non-concordant in $\mathcal{C}$ by Ozsv{\'a}th-Szab{\'o}'s $\tau$ invariant \cite{OSknots} and Rasmussen's $s$ invariant \cite{Rasmussen-thesis}. 

In \cite{Yasui_2015}, Yasui gives a condition on knots such that their images under the Mazur pattern and under a second closely related pattern gives a pair of knots with the same 0-surgery which are not concordant for any orientations. 


In this paper, we will expand on this work in the smooth category. We show the answer to Question 1.1 is ``No"  even when one knot in the pair is torsion in $\mathcal{C}$. 
Interestingly, most of the smooth concordance invariants coming from the knot Floer package vanish for our pair of knots, and both knots are rationally slice. We summarize our results thusly:

\begin{theorem}
\label{theorem:main}
There exist rationally slice knots whose zero surgeries are smoothly $\Z$-homology cobordant rel meridians, but which are not smoothly concordant, and furthermore have different orders in $\mathcal{C}$.
\end{theorem}

The figure eight knot, shown in Figure \ref{fig8}, is negative amphicheiral and is not smoothly slice, so it represents a 2-torsion element of $\cC$. A proof by Cochran, based on earlier work of Fintushel and Stern \cite{FS}, showed that the figure eight knot is rationally slice. The figure eight knot will be the finite-order knot in our pair, and we will recall the satellite construction to build our second knot.

\begin{figure}[ht!]
\centering
\includegraphics[width=35mm]{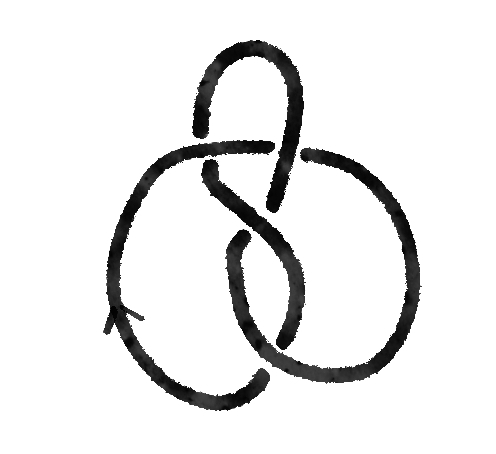}
\caption{The figure eight knot, $4_{1}$} \label{fig8}
\end{figure}

Let $P$ be a knot embedded in $S^1 \times D^2$, let $l$ be the longitude given by $S^{1} \times x$ for $x \in \partial D^{2}$, and let $\nu(K)$ be a tubular neighborhood of $K$. Define $i:S^{1} \times D^{2} \rightarrow \nu(K)$ as the map that sends $l$ to $\lambda_{K}$, the Seifert longitude of $K$. Then, the image of $i(P)$ gives the {\em satellite knot} $P(K)$, with {\em companion} $K$ and {\em pattern} $P$.

The maps given by satellite constructions descend to well-defined functions on $\cC$, and thus satellite operators are useful in the study of the group structure of $\cC$. Of particular interest is knots with winding number $\pm 1$, where the {\em winding number} is computed as the algebraic intersection number of $P$ with a meridional disc in $S^1 \times D^2$. Such knots have historically been of interest due to their use in constructing Mazur $4$-manifolds, and have been extensively studied in such papers as \cite{Cochran_2014} and \cite{Cochran_2013}. We will focus on one specific winding number $+1$ pattern, the Mazur pattern, depicted in Figure \ref{mazurinST}.

\begin{figure}[ht!]
\centering
\includegraphics[width=50mm]{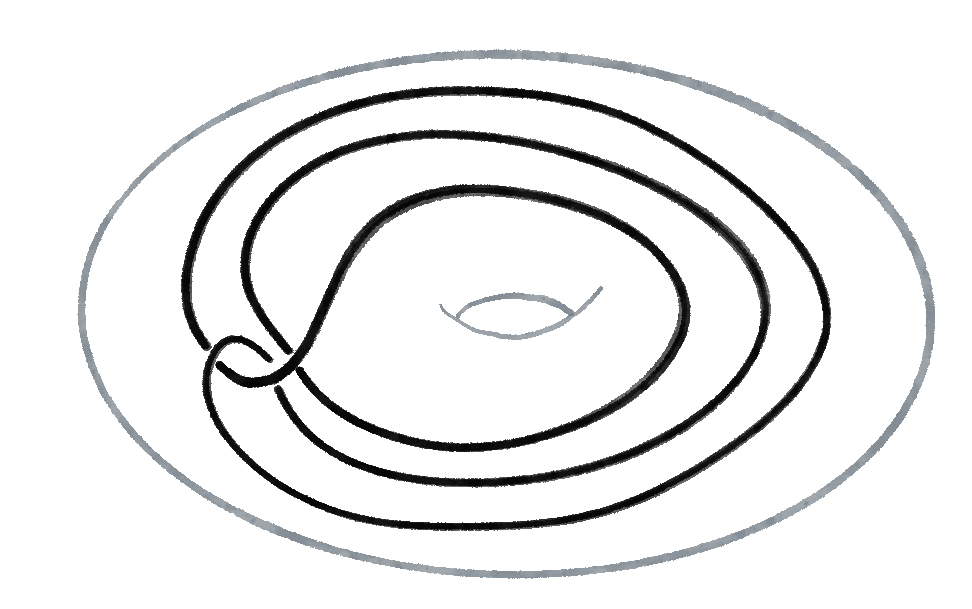}
\caption{The Mazur pattern in the solid torus.} \label{mazurinST}
\end{figure}

We note that the figure eight has Thurston-Bennequin number sufficiently negative to disqualify it from satisfying the conditions outlined for the construction in \cite{Yasui_2015}.

Throughout this paper, we will borrow notation from Levine in \cite{Lev14}. In particular, we let $Q$ denote the Mazur pattern and let $X(L_Q)$ denote the complement of the two-component link $L_Q$, shown in Figure \ref{mazurlink}. The figure eight knot $4_1$ is rationally slice; then $Q(4_1)$ will likewise be rationally slice as $Q(U)$, where $U$ denotes the unknot, is unknotted in $S^3$.

\begin{figure}[ht!]
\centering
\includegraphics[width=40mm]{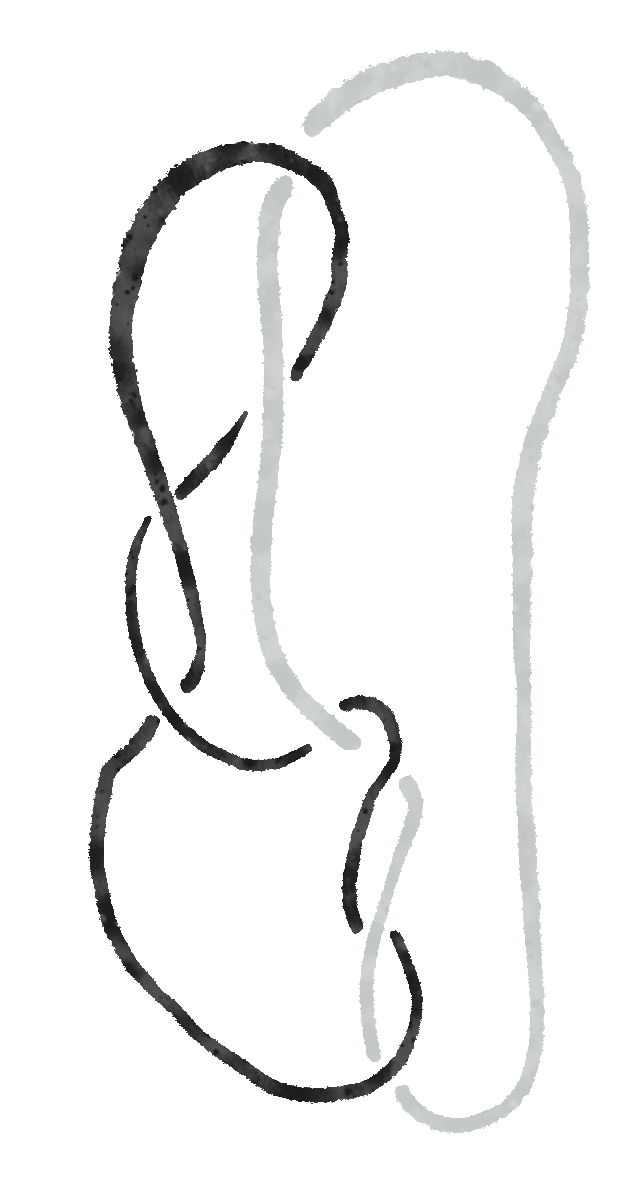}
\caption{$L_{Q}$, the link of the Mazur pattern (black) with the unkot (grey).}
\label{mazurlink}
\end{figure}

Because the Mazur pattern has winding number $+1$, we can appeal to Corollary 2.2 in \cite{Cochran_2013} to conclude that for any knot $K$ in $S^3$, the $0$-surgeries of $K$ and $Q(K)$ are smoothly $\Z$-homology cobordant rel meridians. In particular:

\begin{proposition} 
\label{prop:myknots}
Let K be the figure eight knot in $S^3$. Then, the 0-framed surgery manifolds $S^3_0(K)$ and $S^3_0(Q(K))$ are homology cobordant rel meridian, but $K$ and $Q(K)$ are not smoothly concordant. In particular, $K$ has order two and $Q(K)$ has infinite order in $\cC$.
\end{proposition}

In Section \ref{sec:HF}, we will provide the necessary relevant background in knot Floer homology \cite{OSknots} and bordered Floer homology \cite{Lipshitz_2018}. In Section \ref{sec:3comp}, we will use the bordered Floer homology package to compute the knot Floer complex of $Q(K)$ by first computing the bordered invariants of its complement in $S^3$. We make use of the bimodule $\CFDDh(X(L_Q))$ computed by Levine in \cite{Lev14}, which corresponds to the complement of the Mazur pattern $Q$ in the solid torus, to perform the following $A_\infty$-tensor product:

$$\CFAh(S^3 \smallsetminus \text{nbd}(4_1)) \hspace{.25cm} \widetilde{\otimes}_{A_\infty} \hspace{.25cm} \CFDDh(X(L_Q)) \hspace{.25cm} \simeq \hspace{.25cm} \CFDh(S^3 \smallsetminus \text{nbd}(Q(4_1)))$$

Then, using an algorithm outlined in \cite{Lipshitz_2018} and appealing to the formal properties of knot Floer homology, we convert $\CFDh(S^3 \smallsetminus$ nbd($Q(4_1)$)) into  $\CFKUV(Q(4_1))$ over the ring $\cR = \mathbb{F}_2[\cU,\cV]$. The complex we compute contains the data of $\CFKi(Q(4_1))$ modulo some bigrading information; the fact that knot Floer homology categorifies the Alexander polynomial $\Delta_{K}(t)$ and detects the genus of $K$ allows us to ascertain relative gradings of our basis elements to the extent that we require for the next and final step. In Section \ref{sec:4invol}, we use the involutive knot Floer homology package \cite{HendricksManolescu} to partially compute the almost $\iota_K$-complex of $Q(4_1)$. Once we have this partial computation in hand, in Section \ref{sec:5final} we make use of an algebraic notion of equivalence to make an argument similar to the proof of Theorem 1.3 in \cite{HKPS_2020} to prove that $Q(4_1)$ cannot have finite order in $\cC$. Demonstrating this fact gives us a proof of Proposition \ref{prop:myknots}which immediately implies Theorem \ref{theorem:main}.

Considering the results of this paper, one might find it useful to ask what happens when the Mazur pattern is iterated on the figure eight knot. Based on some further computations we are able to make the following claim:

\begin{Conjecture} \label{conj}
The family of knots $Q^{n}(4_1)$, where $n$ denotes iterations of the Mazur pattern, gives a family of rationally slice knots which are linearly independent in $\mathcal{C}$. Within this family, $\forall \hspace{.1cm} n \geq 1$, $\hspace{.1cm} Q^{n}(4_1)$ has infinite order in $\mathcal{C}$.
\end{Conjecture}

Proof of this conjecture is closely related to the work in this paper, and will appear in an upcoming paper. It follows immediately from Conjecture \ref{conj} that $Q^{n}(4_1)$ for $n \geq 0$ gives an infinite family of rationally slice knots that pairwise have 0-surgery manifolds which are all homology cobordant relative to meridian, but no two knots in this family are smoothly concordant. Proof of this conjecture is also supported by recent work of Hedden and Pinz{\'o}n-Caicedo in \cite{Hedden_2021}.

We set some conventions here that we will abide by for the duration of this paper, unless otherwise indicated. We will assume all knots and knot complements have framing $0$; that is, have framing given by a canonical meridian and a $0$-framed longitude. We let $\F$ denote the field of two elements $\mathbb{Z}/2\mathbb{Z}$, and we use the notation $\CFKUV{(K)}$ for the knot Floer complex of $K$ over the ring $\cR$.
\vspace{.2in}

\noindent \textbf{Acknowledgments:}
The author would like to thank Jennifer Hom for her guidance and support throughout the course of this paper, and throughout the past five years. The author would also like to thank JungHwan Park for helpful conversations, and Arunima Ray for helpful comments on a draft of this paper.
\section{Background: knot and bordered Floer homology}
\label{sec:HF}

We assume familiarity with the Heegaard Floer homology package of Ozsv{\'a}th-Szab{\'o} \cite{HF}. We will provide here an overview of knot Floer homology, defined by Ozsv{\'a}th-Szab{\'o} \cite{OSknots} and independently by J. Rasmussen \cite{Rasmussen-thesis}, and bordered Heegaard Floer homology, defined by Lipshitz-Ozsv{\'a}th-Thurston \cite{Lipshitz_2018}. Section 4 will provide the necessary background on involutive knot Floer homology, defined by Hendricks and Manolescu \cite{HendricksManolescu}. For a more in-depth consideration of any of these theories, we direct the reader to the papers cited, or to the following helpful application and survey papers: \cite{HomSurvey}, \cite{BFHnotes}, and \cite{Hendricks_2019}.

\subsection{Knot Floer Homology}
Knot Floer homology is a Heegaard Floer theory for null-homologous knots embedded in 3-manifolds. We will restrict our attention to knots in $S^3$. To such an oriented knot $K$, we associate a doubly pointed Heegaard diagram $\cH = (\Sigma, \alpha, \beta, w, z)$, with base points $w$ and $z$, where $\Sigma$ is the genus-g closed, oriented Heegaard surface. This surface is decorated with \textbf{$\alpha$}, a g-tuple of $\alpha$-circles, and \textbf{$\beta$}, a g-tuple of $\beta$-circles, and these attaching circles satisfy a series of intersection conditions amongst same sets and between the two sets. Then, the {\em knot Floer complex} is the chain complex we build using the intersection data of two Lagrangian submanifolds $\bbT_{\alpha} = \alpha_{1} \times \alpha_{2} \times \cdots \times \alpha_{g}$ and $\bbT_{\beta} = \beta_{1} \times \beta_{2} \times \cdots \times \beta_{g}$ in the symmetric product of $\Sigma_{g}$. Intersection points between the two tori give generators of our complex, and the differential counts the pseudoholomorphic discs with boundary on $\bbT_{\alpha} \cup \bbT_{\beta}$ that connect these points. The knot Floer complex of $K$ is a $\mathbb{Z} \oplus \mathbb{Z}$-filtered chain complex, whose filtered chain homotopy type is a knot invariant of $K$.

We will use a variation on the full knot Floer complex described in \cite{OSknots} and \cite{Rasmussen-thesis}. In this setting, the {\em knot Floer complex over} $\cR$, denoted $\CFKUV(K)$, is the complex finitely and freely generated over the polynomial ring $\mathcal{R} = \mathbb{F}[\cU,\cV]/ (\cU\cV)$.

Given a generator $\textbf{x} \in \bbT_{\alpha} \cap \bbT_{\beta}$, our chain complex admits a bigrading given by: 

$$\text{gr}(\textbf{x}) = (\text{gr}_{\mathcal{U}}(\textbf{x}), \text{gr}_{\mathcal{V}}(\textbf{x}))$$

\noindent Multiplication by $\mathcal{U}$ lowers the $\mathcal{U}$-grading by 2, multiplication by $\mathcal{V}$ lowers the $\mathcal{V}$-grading by 2, and the differential lowers both gradings by 1. 
Now, given two generators \textbf{x}, \textbf{y} $\in \mathbb{T}_{\alpha} \cap \mathbb{T}_{\beta}$, let $\pi_2(\textbf{x},\textbf{y})$ be the space of homotopy classes of discs connecting \textbf{x} and \textbf{y}. Then, we can define a relative bigrading on $\CFKUV(K)$ as follows:

$$\text{gr}_{\mathcal{U}}(\textbf{x}) -\text{gr}_{\mathcal{U}}(\textbf{y}) = \mu(\phi) - 2n_w(\phi)$$
$$\text{gr}_{\mathcal{V}}(\textbf{x}) -\text{gr}_{\mathcal{V}}(\textbf{y}) = \mu(\phi) - 2n_z(\phi)$$

\noindent
where $\phi \in \pi_2(\textbf{x}, \textbf{y})$, and $\mu(\phi))$ is the Maslov index of this homotopy class. 
\vskip.1in

The set of generators of $\CFKUV{(K)}$ admits a $\mathbb{Z}\oplus \mathbb{Z}$-bigrading $(M,A)$ with {\em M} the {\em Maslov grading} equivalent to the $\mathcal{U}$-grading as described above, and {\em A} the {\em Alexander grading} defined as: 

$$A = \frac{1}{2}(\text{gr}_{\mathcal{U}} - \text{gr}_{\mathcal{V}})$$
\vspace{.05cm}

\noindent We note that it is possible in general to pin down absolute Maslov and Alexander gradings for the generators, but we will only need to make use of absolute Alexander and relative Maslov gradings for our purposes.

We recall some additional properties of Knot Floer homology. The knot Floer complex of a knot detects genus \cite{OSknots} in the following sense. Let $g(K)$ denote the genus of the knot $K$; then

$$g(K) = \text{max}\{i \in \mathbb{Z} : \HFKh(K, i) \neq 0\}$$
\vskip.1in
\noindent Additionally, knot Floer homology categorifies the Alexander polynomial; that is, the graded Euler characteristic of the knot Floer homology of a knot $K$ gives $\Delta_{K}(t)$, the Alexander polynomial of $K$, up to multiplication by a unit.

\subsection{Bordered Floer homology}
Bordered Floer homology is a version of Heegaard Floer theory developed by Lipshitz-Ozsv{\'a}th-Thurston \cite{Lipshitz_2018} for 3-manifolds with boundary. We equip the boundary surface with a basepoint and choose a parameterization of its handlebody decomposition. We then associate to it a differential graded algebra, which we will denote $\mathcal{A}$, which is determined by $\partial M$ equipped with this data. We suppress from notation the choice of parameterization of $\partial M$, and note that our invariants will be independent of this choice up to homotopy equivalence. 

Throughout this paper, all 3-manifolds with boundary are knot or link complements in $S^3$ and thus all boundaries are toroidal. In this case, $\mathcal{A}$ is called the {\em torus algebra}, which we can denote specifically as $\mathcal{A}(\bbT)$. The ring of idempotents of the torus algebra is given by $\mathcal{A}(\mathcal{I}) = (\iota_0, \iota_1)$ where \textbf{1}$= \iota_{0} + \iota_{1}$ gives the unit of $\mathcal{A}$, and we provide an $\mathbb{F}_2$-basis for the algebra: $$\{ \rho_1, \hspace{.1cm} \rho_2, \hspace{.1cm} \rho_3, \hspace{.1cm} \rho_{12},\hspace{.1cm} \rho_{23}, \hspace{.1cm} \rho_{123} \}$$
We describe multiplication in this algebra using the following quiver:

\begin{center}
\begin{tikzpicture}
[>=stealth,
   shorten >=1pt,
   node distance=3cm,
   on grid,
   auto,
   every state/.style={draw=black!60, fill=black!5, very thick}
  ]

 \node[draw=none, fill=none] (mid)               {$\iota_0$};
 \node[draw=none, fill=none] (right) [right=of mid] {$\iota_1$};
 
\path[->]
(right) edge    node[swap]    {$\rho_2$}  (mid)
;

\path[->, draw]
(mid) to [in=120, out=60]    node  {$\rho_1$} (right)
;
\path[->, draw]
(mid) to [in=240, out=300]   node  {$\rho_3$}  (right)

;
\end{tikzpicture}
\end{center}
\noindent The nonzero products of this algebra are given as $\rho_1\rho_2 = \rho_{12}$, $\rho_2\rho_3 = \rho_{23}$, and $\rho_1\rho_2\rho_3 = \rho_{123}$. All other products are zero.

Note that this algebra is 8-dimensional over $\mathbb{F}_2$ and is in fact just a graded algebra; it has no differential. Note also that, as described in the quiver, there are compatibility conditions with the idemptotents; e.g., we have that $\rho_{1} = \iota_{0}\rho_{1}\iota_{1}$, $\rho_{2} = \iota_{1} \rho_{2} \iota_{0}$, etc. Further consideration of the torus algebra is given in Section 11.1 of \cite{Lipshitz_2018}.

With our algebra in hand, we can compute the {\em bordered Floer invariants} of $M$, which are differential graded modules over $\mathcal{A}$. The first type of invariant is called a {\em type D structure}, denoted $\CFDh(M)$, which is a left differential module over $\mathcal{A}$ equipped with a structure map:

$$\delta^{1}: \CFDh(M) \rightarrow \mathcal{A} \otimes \CFDh(M)$$
\vskip.1in
\noindent that satisfies the condition:

$$(\mu_{1} \otimes \text{id}_{\CFDh(M)}) \circ (\text{id}_{\mathcal{A}} \otimes \delta_{1})\circ (\delta_{1}) + (\mu_{2} \otimes \text{id}_{\CFDh(M)}) \circ \delta_{1} = 0$$

\noindent where $\mu_{2}: \mathcal{A} \otimes \mathcal{A} \rightarrow \mathcal{A}$ is (associative) multiplication on $\mathcal{A}$, and $\mu_{1}$ is the differential, in the cases it does not vanish.

The second type of invariant is called a {\em type A structure}, denoted $\CFAh(M)$, which is a right $\mathcal{A}_{\infty}$ module over $\mathcal{A}$ equipped with multiplication maps:

$$m_{i+1}: \CFAh(M) \otimes \underbrace{\mathcal{A} \otimes \dots \otimes \mathcal{A}}_{\text{i copies}} \rightarrow \CFAh(M)$$ 
\vskip.1in
\noindent These $m_i$ maps satisfy a series of compatibility conditions, called the $\mathcal{A}_\infty$ relations. For any element $x \in \CFAh(M)$ and for $a_1, \cdots a_n \in \mathcal{A}$, we have that:

\[
\begin{aligned}
0 &= \sum_{i=0}^{n}m_{n-i+1}(m_{i+1}(x \otimes a_{1} \otimes \cdots \otimes a_{i}) \otimes a_{i+1} \otimes \cdots \otimes a_n) \\
 &+ \sum_{i=1}^{n}m_{n+1}(x \otimes a_{1} \otimes \cdots \otimes a_{i-1} \otimes d(a_i) \otimes a_{i+1} \otimes \cdots \otimes a_n) \\
&+ \sum_{i=1}^{n-1}m_{n}(x \otimes a_1 \otimes a_{i}a_{i+1} \otimes a_{i+2} \otimes \cdots \otimes a_{n}) 
\end{aligned}
\]

These modules are shown to recover the closed Heegaard Floer invariants via a series of pairing theorems. When a closed three-manifold $Y$ is the union of two three-manifolds with boundary $M_{1}$ and $M_{2}$ identified along their boundaries, then the Heegaard Floer complex $\HFh(Y)$ can be recovered via the pairing theorems outlined in \cite{Lipshitz_2018}.  $\HFh(Y)$ is homotopy equivalent to the $A_{\infty}$ tensor product of the type A structure of $Y_i$ with the type D structure of $Y_{i+1 \hspace{.1cm} \text{mod} \hspace{.1cm} 2}$.
When our modules satisfy appropriate conditions of boundedness, we can use the box tensor product $\boxtimes$ in place of the $A_{\infty}$ tensor product $\widetilde{\otimes}$. We describe the pairing theorem, formalized in the Theorem 1.3 of \cite{Lipshitz_2018}, via the following homotopy equivalence:

$$\CFAh(M_1) \hspace{.2cm} \boxtimes \hspace{.2cm} \CFDh(M_2) \hspace{.2cm} \simeq \hspace{.2cm} \CFh(M_1 \cup_{h} M_2)$$ 
where $h$ is an orientation-reversing homeomorphism, $h: \partial M_1 \rightarrow \partial M_2$.

The bordered Floer invariants generalize for connected, oriented three-manifolds with $n$ connected boundary components, where we choose an assignment of a type A or type D structure for each boundary component. For example, when $n=2$, we have three equivalent bimodules: $\CFDDh, \CFAAh,$ or $\CFDAh$. Where there are two boundary components, we will differentiate between the two copies of the torus algebra as $\mathcal{A}_{\rho}$ with basis elements $\rho_i$ and $\mathcal{A}_{\sigma}$ with basis elements $\sigma_j$. In this paper, we will not consider manifold with more than two boundary components. The earlier pairing theorems extend across n-modules; relevant to this paper there is, for $M_1$ a manifold with single boundary component and $M_2$ a manifold with two boundary components, an $\mathcal{A}_\infty$-homotopy equivalence [Theorem 11, \cite{Lipshitz_2011}]: 

$$\CFAh(M_1) \hspace{.2cm} \widetilde{\otimes}_{A_\infty}\hspace{.2cm} \CFDDh(M_2) \hspace{.2cm} \simeq \hspace{.2cm} \CFDh(M_1 \cup_{h} M_2)$$

\noindent where we match the symbols of the algebras to specify which boundary component of $M_2$ is paired with $\partial M_1$. We note that this tensor product only makes sense when the boundary components being matched are surfaces of the same genus. As we will see in Section 3, performing such a computation gives us something homotopy equivalent to a differential module, which we can easily simplify until it more clearly resembles a type D structure. 
\vspace{.1cm}

A distinct advantage of working over the ring $\mathcal{R} = \mathbb{F}[\cU, \cV](\cU\cV)$
is that $\CFKUV(K)$ is equivalent to the bordered invariant $\CFDh(S^3 \smallsetminus nbd(K))$ \cite{Lipshitz_2018}. a straightforward algorithm tells us how to convert between the knot Floer complex of a knot and the bordered Floer module associated to its complement in $S^3$, which is described in Theorem A.11 of \cite{Lipshitz_2018}. As it is highly relevant to this algorithm, we remind the reader that we have assumed a 0-framing on all knots and knot complements.

In section 3 we will provide a detailed overview of performing the tensor product described above; a reader interested in thoroughly understanding the intricacies of such a computation and the underlying algebraic theory is encouraged to consult \cite{Lipshitz_2011}, \cite{Lipshitz_2015}, and Chapter 2 of \cite{Lipshitz_2018}.
\section{The Knot Floer complex of Q(K)}
\label{sec:3comp}

For our computation we will use a model for the $\mathcal{A}_\infty$ tensor product outlined in Section 2.4 of \cite{Lipshitz_2018}. We recall here some conventions and notations.

\begin{definition}[\cite{Lipshitz_2018}, Definition 2.22]
Let $\mathcal{A}$ be an $\mathcal{A}_\infty$ algebra and $N$ a graded left $\mathcal{A}(\mathcal{I})$-module with structure map $\delta^{1}: N \rightarrow \mathcal{A} \otimes N$. We can build higher order maps $\delta^{n}: N \rightarrow \mathcal{A}^{\otimes n} \otimes N $ 
by iterating $\delta^{1}$ inductively. 
Then, $\delta^{1}$ is said to be {\em bounded} if for every $\textbf{x} \in N$, $\delta^{k}(\textbf{x}) = 0$ for $0 \lll k \in \mathbb{Z}.$
\end{definition}
\noindent We then define a bounded type D structure: 

\begin{definition}
Let $(N, \delta_1)$ be a type D structure over $\mathcal{A}$ with base ring $\mathcal{A}(\mathcal{I})$. Then, $(N, \delta_1)$ is said to be {\em bounded} if $\delta^{1}$ is bounded as defined in Definition 3.1.
\end{definition}

\noindent Definitions 3.1 and 3.2 generalize to the case of bimodules in the following way. Let $N$ be a type $DD$ structure over the differential graded algebras $\mathcal{A}$ and $\mathcal{B}$. Then, $N$ comes with a structure map 

$$\delta^{1}: N \rightarrow \mathcal{A}  \otimes N \otimes \mathcal{B}$$

Higher order maps $\delta^{n}$ can be computed inductively as carefully described in \cite{Lipshitz_2015}, with $$\delta^{j}: N \rightarrow \mathcal{A}^{\otimes j} \otimes N \otimes \mathcal{B}^{\otimes j}$$ 

\noindent Then, a(n operationally) bounded type $DD$ structure can be defined analogously to Definition 3.2. We note that the codomain of the structure map suggests notions of left- and right-boundedness; in our case, our module will be both left- and right-bounded simultaneously, hence we will just call it bounded.

\noindent Finally, a bit more convention must be set up:

\begin{definition}[\cite{Lipshitz_2018}, Definition 11.3] 
Let $V_0 \oplus V_1$ be a $\mathbb{F}$-graded vector space. Then, we can define {\em coefficient maps} $D_{\mathcal{I}}$  which are a collection of maps

$$D_{\mathcal{I}}: (V_0 \oplus V_1)^{(i_0 - 1) \hspace{.05cm} \text{mod} \hspace{.05cm} 2} \rightarrow (V_0 \oplus V_1)^{(i_n) \hspace{.05cm} \text{mod} \hspace{.05cm} 2}$$
\noindent where $\mathcal{I}$ is a sequence of consecutive integers $\{i_0, \dots i_k\} \subset \{ 1, 2, 3\}$.
\end{definition}

\vskip.1in
\noindent The way this is interpreted is as follows. Given some portion of a type D structure diagram:
\begin{center}
\begin{tikzpicture}
[>=stealth,
   shorten >=1pt,
   node distance=2cm,
   on grid,
   auto,
   every state/.style={draw=black!60, fill=black!5, very thick}
  ]  
\node[draw=none,fill=none]      (mid)                                  {x};
\node[draw=none,fill=none]      (right)       [right=of mid]              {y};
\node[draw=none,fill=none]      (left)          [left=of mid]              {z};

\path[->]
(mid)     edge        node        {$D_{i_0i_1i_2}$}    (right)
(mid)   edge        node[swap]            {$D_{j_0j_1}$} (left)
;\
\end{tikzpicture}
\end{center}

\noindent we read this as:  $$\delta^{1}(x) = \rho_{j_0j_1} \otimes z + \rho_{i_0i_1i_2} \otimes y$$
\vskip.2in
The purpose of introducing this notation is to allow us to loosely describe how our tensor product will be performed. We note that the following definition is given in \cite{Lipshitz_2018} in greater generality; we restrict to the relevant case for our purposes. 

\begin{definition}[\cite{Lipshitz_2018}, Definition 2.26] Let $\mathcal{A}$ and $\mathcal{B}$ be differential graded algebras, and let $\mathcal{M}$ be a right $A_\infty$ module over $\mathcal{B}$. Let $(\mathcal{N}, \delta^{1})$ be a bounded type DD structure over $\mathcal{A}$ and $\mathcal{B}$ with $\delta^{1}: \mathcal{N} \rightarrow \mathcal{A} \otimes \mathcal{N} \otimes \mathcal{B}$. Then, we can form the differential graded module $\mathcal{M} \boxtimes \mathcal{N}$ called the {\em box tensor product}. Its differential is given by, for all $x \in \mathcal{M}$ and $y \in \mathcal{N}$,

$$ \partial^{\boxtimes} (x \otimes y) = \sum_{\mathcal{I}_j} m_{i+1}(x \otimes \rho_{\mathcal{I}_1} \otimes \cdots \rho_{\mathcal{I}_n}) \otimes (D_{\mathcal{I}_n} \circ \cdots \circ D_{\mathcal{I}_1})(y)$$ 

\noindent where all $\mathcal{I}_j$ are strictly increasing sequences of integers in $\{1, 2, 3\}$. We note that the box tensor product involving a type DD structure is not precisely associative; however it is associative up to homotopy equivalence.
\end{definition}

Recall that type A and type D structures each admit a splitting as a direct sum $V_0 \oplus V_1$ of $\mathbb{F}_2$-vector spaces, with the vector space $V_j$ containing the generators with non-zero $\iota_j$ action. Then, two generators will be tensored only if they are both in the vector space $V_0$ (resp. $V_1$) corresponding to the idempotent $\iota_0$ (resp. $\iota_1$) in the direct sums of their respective structures.

In \cite{Lipshitz_2018}, the authors show that when $\mathcal{M}$ or $\mathcal{N}$ are bounded, then $\mathcal{M} \widetilde{\otimes} \mathcal{N}$ is homotopy equivalent to $\mathcal{M} \boxtimes \mathcal{N}$. As it has several computational advantages, we will use $\boxtimes$ for our calculation.

To illustrate how the box tensor works in practice we give a quick example calculation here. Let $\mathcal{M}$ and $(N, \delta^{1})$ be as described in Definition 3.4. Let $\rho_i$ denote the basis elements of a copy of the torus algebra $\mathcal{A}$, and $\sigma_j$ denote the basis elements of another copy of the torus algebra $\mathcal{B}$. Let $x \in \mathcal{M}$ and suppose its multiplication maps are given by

$$m_2(x, \sigma_1) = w$$
$$m_4(x, \sigma_3, \sigma_2, \sigma_1) = z$$

\vskip.2in
\noindent Let $q \in N$ and let the relevant components of the image of the type D structure map be given by:

$$\delta^{1}(q) = \rho_2 \otimes r + \sigma_{1}\rho_{123} \otimes s + \sigma_{3} t$$
$$\delta^{1}(t) = \sigma_{2} \otimes u$$
$$\delta^{1}(u) = \sigma_{1} \otimes v$$

\noindent Then, 
$$\partial^{\boxtimes}(x \otimes q) = \sigma_{123} \otimes (w \otimes s) + (z \otimes v) + \sigma_2 \otimes (x \otimes r)$$
\vskip.2in

We note that this box tensor product operation results in unlabeled edges of our diagram, which are in general not an issue, but which are useful to remove before attempting to produce the knot Floer complex. We will require some edge cancellation lemmas to address these unlabeled edges that arise, and we will describe those presently. For simplicity, we will henceforth use the notation $X_K$ to denoted the complement of the knot $K$ in $S^3$, with 0-framing. 

\vskip.2in
\noindent \textbf{\textit{Partial Tensor Product Computation:}} In order to perform our desired box tensor product, we will reference the bimodule computed by Levine, denoted as $\CFDDh(X(L_{Q}))$ and described in Theorem 3.4 of \cite{Lev14}. In keeping with the notation used by the author, $\mathcal{A}_\sigma$ is assigned to the boundary component of the complement of the pattern and $\mathcal{A}_\rho$ is assigned to the boundary component of the complement of the unknot.

We will tensor the type A structure of the figure eight complement with the type DD structure of the Mazur link complement. We can use the algorithm outlined in section A.4 of \cite{Lipshitz_2018} to build $\CFDh(X_{4_1})$, then use the process described in \cite{Hedden_2016} to convert the type D structure to a type A structure. The type A structure of the figure eight complement is pictured in Figure \ref{fig:TypeA}. Note that the dotted edges denote valid multiplication maps of our type A structure, but they will not contribute anything to our specific $\partial^{\boxtimes}$ since there is not, for example, a sequence of coefficient maps $D_1 \circ D_2 \circ D_{123}$  in $\CFDDh(X(L_Q))$ for any generators associated to $\iota_0^{\sigma}.$ Note also that we labeled our edges with $\sigma_i$ basis elements, to indicate to which boundary component of the link complement $X(L_{Q})$ we are gluing the knot complement $X_{4_1}$.

\begin{figure}
\resizebox{16cm}{!}{

\begin{tikzpicture} 
[>=stealth,
   shorten >=1pt,
   node distance=3cm,
   on grid,
   auto,
   every state/.style={draw=black!60, fill=black!5, very thick}
  ]
 \node[draw=none,fill=none] (mid)                             {w};
 \node[draw=none,fill=none] (left)      [left=of mid]          {a};
 \node[draw=none,fill=none] (right)     [right=of mid]          {b};
 \node[draw=none,fill=none] (right2)    [below=of right]        {x};
 \node[draw=none,fill=none] (right3)    [below=of right2]       {c};
 \node[draw=none,fill=none] (mid3)      [left=of right3]        {y};
 \node[draw=none,fill=none] (left2)     [below=of left]         {z};
 \node[draw=none,fill=none] (left3)     [below=of left2]        {d};
 
\path[->]
    (mid)       edge                node[swap]          {$\sigma_2$}                            (left)
    (right)     edge                node[swap]          {$\sigma_1$}                            (mid)
    (right)     edge                node                {$\sigma_3$}                            (right2)
    (right3)    edge                node[swap]          {$\sigma_{3}, \sigma_{2}, \sigma_{1}$}  (right2)
    (right3)    edge                node                {$\sigma_1$}                            (mid3)
    (mid3)      edge                node                {$\sigma_2$}                            (left3)
    (left3)     edge                node                {$\sigma_{3}, \sigma_{2}, \sigma_{1}$}  (left2)
    (left)      edge                node[swap]          {$\sigma_{3}$}                          (left2)
    (mid)       edge[bend left]     node[swap]          {$\sigma_{23}$}                         (left2)
    (right)     edge[bend left]     node[swap]          {$\sigma_{123}$}                        (left2)
    (mid3)      edge                node[rotate=-45, xshift=1cm]   {$\sigma_{23}, \sigma_{2}, \sigma_{1}$} (left2)
    ;
    
\path[->, dashed]

    (right)     edge[bend right]        node[swap]      {$\sigma_{12}$}                         (left)
    (right3)    edge[bend left]         node            {$\sigma_{12}$}                         (left3)
   
;

\path[->, draw, dashed]
   
   (right3)   to [in=-30, out=145]       node[swap, rotate=-20, xshift=-.5cm]   {$\sigma_{123}, \sigma_2, \sigma_1$}    (left2)

;

\end{tikzpicture}

\tikzset{every loop/.style={min distance=10mm,looseness=10}}
\begin{tikzpicture}

        \node [] (alone) {e} ;

        \path[->, draw] (alone) to [in=10,out=90,loop, distance=1cm] node[auto, rotate=-45, xshift=-.5cm] {$\sigma_{3}, \sigma_{2}$} (alone);

        \path[->,draw] (alone) to [in=0, out=100, loop, distance=4cm] node[auto, rotate=-45, xshift=-.5cm] {$\sigma_{3}, \sigma_{23}, \sigma_{2}$} (alone);

        \path[->,draw] (alone) to  [in=-10,out=110,loop,distance=9cm] node[auto, rotate=-45, xshift=-1cm] {$\sigma_{3}, \sigma_{23}, \sigma_{23}, \sigma_{23}, \dots$} (alone);

        ;
    \end{tikzpicture}
    
}
\caption{The module $\CFAh(X_{4_1}))$ where $4_1$ has 0-framing}
\label{fig:TypeA}
\end{figure}

\noindent Recall that the goal of this section is to provide the result of the box tensor product:

$$ \CFAh(X_{4_1}) \hspace{.2cm} \boxtimes \hspace{.2cm}  \CFDDh(X(L_{Q})) \simeq \CFDh(X_{Q(4_1)})$$
\vskip.1in
\noindent We will not perform all of the painful computations here. Rather, as a means of illustrating the procedure by which this full answer is produced, we will carefully compute here $\partial^{\boxtimes}(a \otimes g_{i})$ for $g_i$ in the set: 

$$\{g_i \in \CFDDh(X(L_Q)) \hspace{.25cm} | \hspace{0.25cm} (\iota_0^{\sigma} \otimes \iota_j^{\rho}) \otimes g_i = g_i \}$$
\vskip.1in
\noindent Namely, we will box tensor $a \in \CFAh(X_{4_1})$ with each generator in the set: $$\{g_2,\hspace{.1cm} g_4,\hspace{.1cm} g_5,\hspace{.1cm} g_9,\hspace{.1cm} g_{10},\hspace{.1cm} g_{16},\hspace{.1cm} g_{21},\hspace{.1cm} g_{23},\hspace{.1cm} g_{25},\hspace{.1cm} g_{27},\hspace{.1cm} g_{29},\hspace{.1cm} g_{32},\hspace{.1cm} g_{34}\}$$
\vskip.1in

\noindent Note that the idempotent action from the other copy of the algebra $\mathcal{A}_\rho$ does not impact how generators are tensored in $\partial^{\boxtimes}.$ A reader disinterested in or otherwise well-acquainted with taking a box tensor product who wishes to skip this explanation is directed to the final full tensor product given in Figures \ref{fig:dotfig} and \ref{fig:bigfig}. 
For simplification, we establish the following convention: we will write $\alpha \otimes g_i$ as simply $\alpha g_i$ for any generator $\alpha \in \CFAh(X_{4_1})$.

It will hopefully be clear to the reader from context where a $\otimes$ sign has been omitted. We note that we will maintain using the $\otimes$ sign between an algebraic element and a generator; we omit it only between two generators.
Following this notation, we are now ready to compute $\partial^{\boxtimes}(ag_i)$.  The only nonzero multiplication map on $a$ is
$m_2(a,\sigma_3) = z$.
The relevant parts of the differential map of the type DD module are given by:
\vskip.2in
\begin{itemize}
\item[] $d(g_2) = \sigma_{3}\rho_3 \otimes g_6$
\item[] $d(g_4) = \rho_2 \otimes g_{21} + \sigma_{3} \otimes g_{26}$
\item[] $d(g_9) = \rho_{23} \otimes g_{25}$
\item[] $d(g_{10}) = \sigma_{3} \otimes g_{33}$
\item[] $d(g_{16}) = \sigma_{3} \otimes g_{22}$
\item[] $d(g_{21}) = \sigma_{3} \otimes g_{15}$
\item[] $d(g_{25}) = \sigma_{3}\rho_{23} \otimes g_8 +  \sigma_{3}\rho_{23} \otimes g_{24}$
\item[] $d(g_{27}) = \rho_3 \otimes g_9$
\item[] $d(g_{29}) = \rho_3 \otimes g_4 + \sigma_{3} \otimes g_{18}$
\item[] $d(g_{32}) = \rho_2 \otimes g_2 + \sigma_{3} \otimes g_{20}$
\item[] $d(g_{34}) = \rho_3 \otimes g_{32} + \sigma_{3} \otimes g_{13}$
\end{itemize}

\begin{figure}
\resizebox{16cm}{!}{

\begin{tikzpicture}
[>=stealth,
   shorten >=1pt,
   node distance=3cm,
   on grid,
   auto,
   every state/.style={draw=black!60, fill=black!5, very thick}
 ]

\node[draw=none,fill=none] (mid)                                {$ag_2$};
\node[draw=none,fill=none] (right)      [right=of mid]          {$zg_6$};
\node[draw=none, fill=none] (mid2)      [below=1cm of mid]          {$ag_4$};     
\node[draw=none, fill=none] (right2)    [right=of mid2]         {$zg_{26}$};
\node[draw=none, fill=none] (mid2')     [below=1cm of right2]         {$ag_{21}$};
\node[draw=none, fill=none] (mid3)      [right=2cm of right]        {$ag_9$};
\node[draw=none, fill=none] (right3)    [right=of mid3]         {$ag_{25}$};
\node[draw=none, fill=none] (mid4)      [below=1cm of mid3]         {$ag_{10}$};
\node[draw=none, fill=none] (right4)    [right=of mid4]         {$zg_{33}$};
\node[draw=none, fill=none] (mid5)      [below=1cm of mid4]         {$ag_{16}$};
\node[draw=none, fill=none] (right5)    [right=of mid5]         {$zg_{22}$};
\node[draw=none, fill=none] (mid6)      [below=1cm of mid5]         {$ag_{21}$};  
\node[draw=none, fill=none] (right6)    [right=of mid6]         {$zg_{15}$};
\node[draw=none, fill=none] (mid7)      [right=2cm of right3]         {$ag_{25}$};
\node[draw=none, fill=none] (right7)    [right=of mid7]         {$zg_8$};
\node[draw=none, fill=none] (mid7')     [below=1cm of right7]         {$zg_{24}$};
\node[draw=none, fill=none] (mid8)      [below=2cm of mid7]        {$ag_{27}$};   
\node[draw=none, fill=none] (right8)    [right=of mid8]         {$ag_9$};
\node[draw=none, fill=none] (mid9)      [below=1cm of mid8]         {$ag_{29}$};  
\node[draw=none, fill=none] (right9)    [right=of mid9]         {$ag_4$};
\node[draw=none, fill=none] (mid9')     [below=1cm of right9]         {$zg_{18}$};
\node[draw=none, fill=none] (mid10)      [right=2cm of right7]       {$ag_{32}$};    
\node[draw=none, fill=none] (right10)    [right=of mid10]       {$ag_2$};
\node[draw=none, fill=none] (mid10')    [below=1cm of right10]        {$zg_{20}$};
\node[draw=none, fill=none] (mid11)      [below=2cm of mid10]      {$ag_{34}$};     
\node[draw=none, fill=none] (right11)    [right=of mid11]       {$ag_{32}$};
\node[draw=none, fill=none] (mid11')    [below=1cm of right11]        {$zg_{13}$};

 \path[->]
 
(mid)            edge               node            {$D_3$}      (right)
(mid2)          edge                node            {}              (right2)
(mid2)          edge[bend right]     node            {$D_2$}      (mid2')
(mid3)          edge                node            {$D_{23}$}   (right3)
(mid4)          edge                node            {}              (right4)
(mid5)          edge                node            {}              (right5)
(mid6)          edge                node            {}              (right6)
(mid7)          edge                node            {$D_{23}$}   (right7)
(mid7)          edge[bend right]    node            {$D_{23}$}   (mid7')
(mid8)          edge                node            {$D_3$}      (right8)
(mid9)          edge                node            {$D_3$}      (right9)
(mid9)         edge[bend right]     node            {}              (mid9')
(mid10)         edge                node            {$D_2$}      (right10)
(mid10)         edge[bend right]    node            {}              (mid10')
(mid11)         edge                node            {$D_3$}      (right11)
(mid11)         edge[bend right]    node            {}              (mid11')

 ;

\end{tikzpicture}

}
\caption{Result of $\partial^{\boxtimes}(a \otimes g_{i})$}
\label{fig:egtensor}
\end{figure}
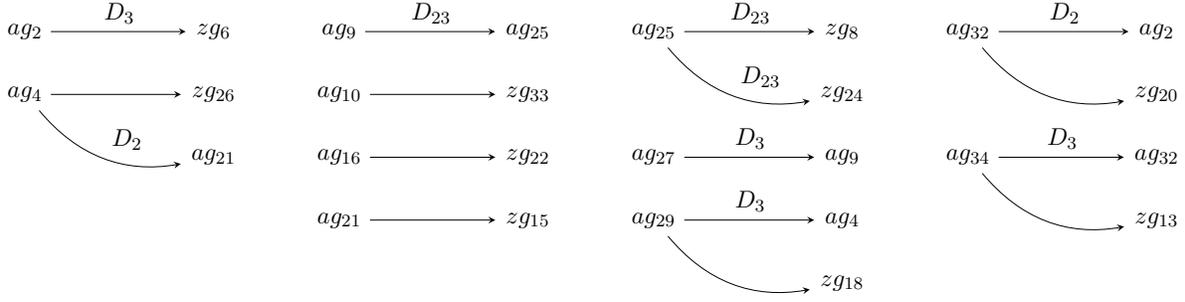
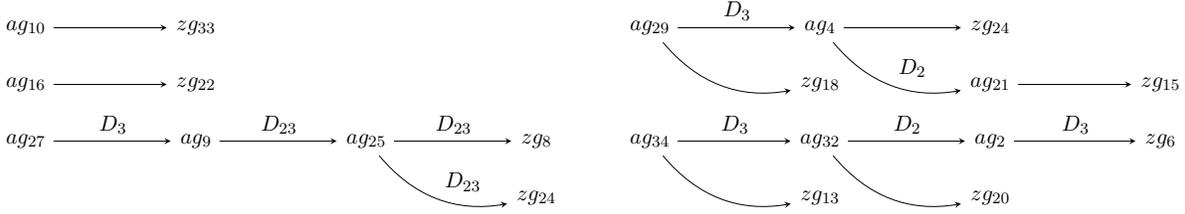
\begin{figure}
\resizebox{16cm}{!}{

\begin{tikzpicture}
[>=stealth,
   shorten >=1pt,
   node distance=3cm,
   on grid,
   auto,
   every state/.style={draw=black!60, fill=black!5, very thick}
 ]

\node[draw=none, fill=none]     (mid)                                   {$ag_{10}$};
\node[draw=none, fill=none]     (right)     [right=of mid]              {$zg_{33}$};
\node[draw=none, fill=none]     (mid2)      [below=1cm of mid]          {$ag_{16}$};
\node[draw=none, fill=none]     (right2)    [right=of mid2]             {$zg_{22}$};
\node[draw=none, fill=none]     (left3)     [below=1cm of mid2]         {$ag_{27}$};
\node[draw=none, fill=none]     (mid3)      [right=of left3]              {$ag_{9}$}; 
\node[draw=none, fill=none]     (right3)    [right=of mid3]              {$ag_{25}$}; 
\node[draw=none, fill=none]     (right3')    [right=of right3]          {$zg_{8}$}; 
\node[draw=none, fill=none]     (right3'')   [below=1cm of right3']     {$zg_{24}$};
\node[draw=none, fill=none]     (left4)     [right=8cm of right]        {$ag_{29}$};
\node[draw=none, fill=none]     (mid4)     [right=of left4]             {$ag_{4}$};
\node[draw=none, fill=none]     (mid4')     [below=1cm of mid4]         {$zg_{18}$};
\node[draw=none, fill=none]     (right4)     [right=of mid4]            {$zg_{24}$};
\node[draw=none, fill=none]     (right4')     [below=1cm of right4]     {$ag_{21}$};
\node[draw=none, fill=none]     (right4'')     [right=of right4']         {$zg_{15}$};
\node[draw=none, fill=none]     (left5)     [below=2cm of left4]        {$ag_{34}$};
\node[draw=none, fill=none]     (mid5)     [right=of left5]             {$ag_{32}$};
\node[draw=none, fill=none]     (mid5')     [below=1cm of mid5]         {$zg_{13}$};
\node[draw=none, fill=none]     (right5)     [right=of mid5]            {$ag_{2}$};
\node[draw=none, fill=none]     (right5')     [below=1cm of right5]     {$zg_{20}$};
\node[draw=none, fill=none]     (right5'')     [right=of right5]        {$zg_{6}$};

 \path[->]
 
(mid)          edge                node                {}                      (right)
(mid2)         edge                node                {}                      (right2)
(left3)         edge                node                {$D_3$}              (mid3)
(mid3)          edge                node                {$D_{23}$}           (right3)
(right3)         edge               node                {$D_{23}$}           (right3')
(right3)        edge[bend right]     node                {$D_{23}$}           (right3'')
(left4)         edge                node                {$D_3$}              (mid4)
(left4)         edge[bend right]    node                {}                      (mid4')
(mid4)          edge                node                {}                      (right4)
(mid4)          edge[bend right]    node                {$D_2$}              (right4')
(right4')       edge                node                {}                      (right4'')
(left5)         edge                node                {$D_3$}              (mid5)
(left5)         edge[bend right]    node                {}                      (mid5')
(mid5)          edge                node                {$D_2$}              (right5)
(mid5)          edge[bend right]    node                {}                      (right5')
(right5)         edge                node                {$D_3$}              (right5'')
;

\end{tikzpicture}

}
\caption{Result of $\partial^{\boxtimes}(a \otimes g_{i})$ after concatenation}
\label{fig:egtensorcon}
\end{figure}

Note that we excluded any generators $g_i \in \CFDDh(X(L_Q))$ for which $\partial^{\boxtimes}(a \otimes g_i) = 0$, whether by mismatched idempotent action or mismatched algebra elements as they will not contribute any edges to the tensor complex. By performing the box tensor product as described in Definition 3.4, we generate the edges shown in Figure \ref{fig:egtensor}. We concatenate edges to arrive at the complex section shown in Figure \ref{fig:egtensorcon}. We continue with the box tensor process for all of the generators in $\CFAh(X_{4_1})$ until we arrive at the tensor complex, pictured in Figures \ref{fig:dotfig} and \ref{fig:bigfig}. Figure \ref{fig:bigfig} gives the result of tensoring the type DD structure with the box component of $\CFAh(X_{4_1})$, and Figure \ref{fig:dotfig} gives the result of tensoring the type DD structure with the lone generator component of $\CFAh(X_{4_1})$.

We note that these are not the complete tensor complexes in a sense; performing the tensor product results in some additional standalone unlabeled edges, but we can just homotope these edges away. In fact, our goal will be to homotope away all unlabeled edges in the two diagrams, so we can more easily produce the knot Floer complex from this data. We accomplish this by utilizing the following edge cancellation lemma, which makes use of the fact that we are working over $\mathbb{F}_2$ to peel away any edges into or out of an unlabelled edge.

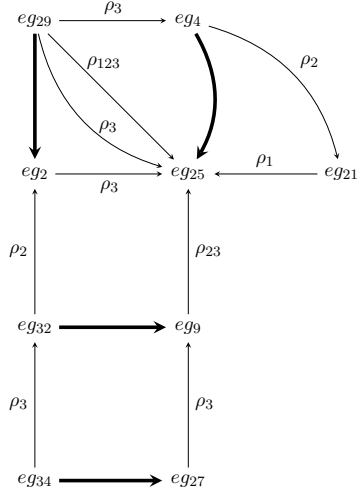
\begin{figure}
    \centering
    \resizebox{5cm}{!}{

\begin{tikzpicture} 
[>=stealth,
   shorten >=1pt,
   node distance=3cm,
   on grid,
   auto,
   every state/.style={draw=black!60, fill=black!5, very thick}
  ]
 \node[draw=none,fill=none] (mid)                             {$eg_{25}$};
 \node[draw=none,fill=none] (left)      [left=of mid]          {$eg_{2}$};
 \node[draw=none,fill=none] (right)     [right=of mid]          {$eg_{21}$};
 \node[draw=none,fill=none] (mid2)      [below=of mid]        {$eg_{9}$};
 \node[draw=none,fill=none] (mid3)    [below=of mid2]       {$eg_{27}$};
 \node[draw=none,fill=none] (left3)      [left=of mid3]        {$eg_{34}$};
 \node[draw=none,fill=none] (left2)     [above=of left3]         {$eg_{32}$};
 \node[draw=none,fill=none] (aleft)     [above=of left]        {$eg_{29}$};
 \node[draw=none,fill=none] (amid)     [above=of mid]        {$eg_{4}$};
 
\path[->]
    (left)       edge               node[swap]          {$\rho_3$}         (mid)
    (right)     edge                node[swap]           {$\rho_1$}      (mid)
    (mid2)      edge                 node[swap]               {$\rho_{23}$}     (mid)
    (mid3)      edge                node[swap]          {$\rho_3$}       (mid2)
    (left3)    edge[line width=2]   node                  {}         (mid3)
    (left3)     edge                node                {$\rho_3$}       (left2)
    (left2)     edge                node                {$\rho_2$}         (left)
    (aleft)     edge[line width=2]  node                {}              (left)
    (aleft)     edge               node[near start, yshift=-.2cm]           {$\rho_{123}$}     (mid)
    (aleft)     edge[bend right]    node[yshift=-.4cm, xshift=.2cm]             {$\rho_3$}         (mid)
    (aleft)     edge                node                {$\rho_3$}         (amid)
    (amid)      edge[bend left, line width=2]     node                {}              (mid)
    (amid)      edge[bend left]    node                {$\rho_2$}         (right)
    (left2)     edge[line width=2]                node                {}              (mid2)
    ;
    
\end{tikzpicture}

}
    \caption{Tensor complex part 1}
    \label{fig:dotfig}
\end{figure}

\begin{figure}
\centering
\input{BIGmodule.tex}
\caption{Tensor complex part 2}
\label{fig:bigfig}
\end{figure}

\begin{table}[hb]
\begin{tabular}{c c c}
\hline
  \sc{Case I:}  &  & \begin{tikzpicture} 
[>=stealth,
   shorten >=1pt,
   node distance=2cm,
   on grid,
   auto,
   every state/.style={draw=black!60, fill=black!5, very thick}
  ]

\node[draw=none, fill=none] (mid)                           {\textbf{$a$}};
\node[draw=none, fill=none] (right)     [right=1.5cm of mid]      {\textbf{$b$}};
\node[draw=none, fill=none] (down)      [below=of mid]      {\textbf{$c$}};
\node[draw=none, fill=none] (diag)      [right=1.5cm of down]     {\textbf{$d$}};
\node[draw=none, fill=none] (symb)      [right=3cm of mid, yshift=-1cm]   {\large \textbf{$\simeq$}};

\node[draw=none, fill=none] (mid2)      [right=4.5cm of mid]  {\textbf{$a+ \rho_{i} \otimes b$}};
\node[draw=none, fill=none] (right2)    [right=of mid2, xshift=-1cm, yshift=-.5cm]     {\textbf{$b$}};
\node[draw=none, fill=none] (down2)     [below=of mid2]              {\textbf{$c+ \rho_{j} \otimes d$}};
\node[draw=none, fill=none] (diag2)     [right=of down2]    {\textbf{$d$}};

\path[->]
(mid)           edge            node[swap]      {$D_{i}$}        (down)
(right)         edge[line width=1.5]            node            {}                  (down)
(right)         edge            node            {$D_{j}$}        (diag)
(right2)        edge[line width=1.5]            node            {}                  (down2)
;

\path[->, draw]
(mid2)  to [out=0, in=90, looseness=1.5]    node    {$D_{ij}$} (diag2)
;
\end{tikzpicture} \\
  \hline
  \sc{Case II:} & &  \resizebox{8cm}{!}{

\begin{tikzpicture} 
[>=stealth,
   shorten >=1pt,
   node distance=2cm,
   on grid,
   auto,
   every state/.style={draw=black!60, fill=black!5, very thick}
  ]

\node[draw=none, fill=none] (midA)                               {\textbf{$a$}};
\node[draw=none, fill=none] (rightA)         [right=of midA]      {\textbf{$b$}};
\node[draw=none, fill=none] (downA)          [below=of midA]      {\textbf{$c$}};
\node[draw=none, fill=none] (diagA)          [right=of downA]     {\textbf{$d$}};
\node[draw=none, fill=none] (symbA)      [right=3.5cm of midA, yshift=-1cm]   {\large \textbf{$\simeq$}};

\node[draw=none, fill=none] (midA2)          [right=5cm of midA]     {\textbf{$a$}};
\node[draw=none, fill=none] (rightA2)        [right=of midA2]         {\textbf{$b$}};
\node[draw=none, fill=none] (downA2)         [below=of midA2]        {\textbf{$c$}};
\node[draw=none, fill=none] (diagA2)         [right=of downA2]       {\textbf{$d + \rho_{k} \otimes a$}};

\path[->]
(midA)       edge[line width=1.5]       node        {}             (downA)
(rightA)     edge[line width=1.5]       node        {}             (diagA)
(rightA)     edge                       node[swap]  {$D_{k}$}   (midA)
(diagA)      edge                       node        {$D_{k}$}   (downA)
(midA2)      edge[line width=1.5]       node        {}             (downA2)
(rightA2)    edge[line width=1.5]       node        {}             (diagA2)
;

\end{tikzpicture}

} \\
  \hline
  \sc{Case III:} & & \resizebox{9cm}{!}{

\begin{tikzpicture} 
[>=stealth,
   shorten >=1pt,
   node distance=2cm,
   on grid,
   auto,
   every state/.style={draw=black!60, fill=black!5, very thick}
  ]  
  
  \node[draw=none, fill=none]   (midB)                          {\textbf{$c$}};
  \node[draw=none, fill=none]   (downB)     [below=2cm of midB]      {\textbf{$b$}};
  \node[draw=none, fill=none]   (rightB)    [right=1.75cm of downB]     {\textbf{$a$}};
  \node[draw=none, fill=none]   (leftB)     [left=1.75cm of downB]      {\textbf{$d$}};
  \node[draw=none, fill=none]   (symbB)     [right=2.5cm of midB, yshift=-1cm]   {\large \textbf{$\simeq$}};

  \node[draw=none, fill=none]   (midB2)   [right=6cm of midB, yshift=-.5cm] {\textbf{$c+ \rho_{i} \otimes a$}};
  \node[draw=none, fill=none]   (downB2)  [below=1cm of midB2]      {\textbf{$b$}};
  \node[draw=none, fill=none]   (rightB2) [right=1.75cm of downB2]     {\textbf{$a$}};
  \node[draw=none, fill=none]   (leftB2)  [left=1.75cm of downB2]      {\textbf{$d+ \rho_{j} \otimes a$}};
  
  \path[->]
  (leftB)        edge                           node        {$D_{j}$}   (downB)
  (rightB)       edge[line width=1.5]           node        {}             (downB)
  (midB)         edge                           node        {$D_{i}$}   (downB)
  (rightB2)      edge[line width=1.5]           node        {}             (downB2)
  ;

 \end{tikzpicture}
 
 } \\
  \hline
  \caption{Our three cases of edge cancellation}
  \label{edgecx}
\end{tabular}
\end{table}

\begin{lemma}\label{lem:edgecx} In our module, let $(v_2, v_1)$ be the unlabeled edge from basis element $v_2$ to basis element $v_1$. Then, for each edge into $v_1$, denoted $(a_j, v_1)$ and labeled by $D_j$ with $j \in \{1, 2, 3, 12, 23, 123\}$, for all edges out of $v_2$, denoted $(v_2, b_i)$ and labeled by $D_i$ with $i \in \{1, 2, 3, 12, 23, 123\}$, we perform the following change of basis:
$$a_{j}^{'} = a_j + \rho_{j} \otimes v_2, \hspace{1in} v_{1}^{'} = v_1 + \rho_{i} \otimes b_i$$

The first change lets us cancel the edge $(a_j, v_1)$, since $char(\mathbb{F}_2) = 2$. Whenever $ij \neq 0$ in $\mathcal{A}$, this will also generate a new edge $(a_j+ D_{j}(v_2), b_i)$ labeled with $D_{ij}$. The second change lets us cancel the edge $(v_2, b_i)$.
Then, by the compatibility conditions satisfied by $\mathcal{A}_{\infty}$ modules, once all of the basis changes are performed, all other edges out of $v_1$ and all other edges into $v_2$ will be eliminated from the diagram. At the end of the process, the unlabeled edge $$(v_2, v_1 + \sum_{i \in I} D_{i}(b_i))$$ can be homotoped away from the diagram. 
\end{lemma}

We provide in Table \ref{edgecx} a diagrammatic description of the three cases in which we use Lemma \ref{lem:edgecx} for our particular diagram. Note that these three cases are all generalized by the above lemma.

We zoom in on a portion of the unsimplified complex and explicitly provide the changes of basis that allow us to isolate and then homotope away our unlabeled edges. In Figure \ref{fig:zoomUnsimp}, we isolate a section of $\CFDh(X_{Q(4_1)})$. Then, in Figure \ref{fig:zoomSimp}, we illustrate the way our basis should be changed to isolate the unlabeled edges.

\begin{figure}[h!]
\begin{tikzpicture}
[>=stealth,
   shorten >=1pt,
   node distance=1.85cm,
   on grid,
   auto,
   every state/.style={draw=black!60, fill=black!5, very thick}
  ]  

\node[draw=none,fill=none]      (11)                        {$xg_{31}$};
\node[draw=none,fill=none]      (12)        [right=of 11]   {$cg_{23}$};
\node[draw=none,fill=none]      (13)        [right=of 12]   {$yg_{8}$};
\node[draw=none,fill=none]      (14)        [right=of 13]   {$yg_{20}$};
\node[draw=none,fill=none]      (15)        [right=of 14]   {$yg_{13}$};
\node[draw=none,fill=none]      (16)        [right=of 15]   {$dg_{27}$};

\node[draw=none,fill=none]      (21)        [below=of 11]   {$cg_{10}$};
\node[draw=none,fill=none]      (24)        [right=of 23]   {$yg_{6}$};
\node[draw=none,fill=none]      (25)        [right=of 24]   {$dg_{25}$};
\node[draw=none,fill=none]      (26)        [right=of 25]   {$dg_{9}$};

\node[draw=none,fill=none]      (31)        [below=of 21]   {$yg_{12}$};
\node[draw=none, fill=none]     (32)        [right=of 31]   {};
\node[draw=none, fill=none]     (33)        [right=of 32]   {};
\node[draw=none,fill=none]      (34)        [right=of 33]   {$yg_{30}$};
\node[draw=none,fill=none]      (35)        [right=of 34]   {$yg_{15}$};

\node[draw=none, fill=none]     (44)        [below=of 34]   {};
\node[draw=none, fill=none]     (45)        [below=of 35]   {};
\node[draw=none, fill=none]     (36)        [below=of 26]   {};

\node[draw=none, fill=none]     (41)        [below=of 31] {};

\path[->]
(12)        edge[line width=1.5]   node            {}              (13)
(15)        edge                 node[swap]     {$D_{3}$}         (14)
(15)        edge[line width=1.5]  node            {}              (16)

(21)        edge[line width=1.5]  node            {}              (11)
(14)        edge                node            {$D_{23}$}      (24)
(16)        edge                node            {$D_{3}$}       (26)

(24)        edge                node            {$D_{2}$}       (34)
(21)        edge[line width=1.5]  node            {}              (31)

(24)        edge[line width=1.5]  node            {}              (25)
(26)        edge                node            {$D_{23}$}      (25)
(35)        edge                node[swap]            {$D_{1}$}       (25)
;

\path[->, loosely dashed, ultra thin]
(45)        edge    node        {}      (35)
(36)        edge    node        {}      (25)
(44)        edge    node        {}      (34)
(41)        edge    node        {}      (31)
(32)        edge    node        {}      (31)
(34)        edge    node        {}      (33)
;

\path[->, draw]
(21)  to [in=230, out=45] node[near start, xshift = 15pt] {$D_{23}$}  (13)
;

\path[->, draw]
(34)  to [in=0, out=160]    node[near end, xshift=4pt]  {$D_{123}$}   (13)
;

\path[->, draw, line width=1.5]
(14)  to [in=135, out=315]  node               {}              (26)
;

\end{tikzpicture}
\caption{Unsimplified portion of $\CFDh(X_{Q(4_1)}$}
\label{fig:zoomUnsimp}
\end{figure}

\begin{figure}[h!]
\begin{tikzpicture}
[>=stealth,
   shorten >=1pt,
   node distance=1.85cm,
   on grid,
   auto,
   every state/.style={draw=black!60, fill=black!5, very thick}
  ]

\node[draw=none,fill=none, text centered, text width=4.5em]       (11)    {$xg_{31}\hspace{.1cm}+\rho_{23}\otimes yg_{8}$\\ $+\hspace{.1cm}yg_{12}$};
\node[draw=none,fill=none]      (12)        [right=of 11]   {$cg_{23}$};
\node[draw=none,fill=none]      (13)        [right=of 12]   {$yg_{8}$};
\node[draw=none,fill=none]      (14)        [right=of 13]   {$yg_{20}$};
\node[draw=none,fill=none]      (15)        [right=of 14]   {$yg_{13}$};
\node[draw=none,fill=none, text centered, text width=4.5em, xshift=0.8cm]      (16)        [right=of 15]   {$dg_{27}+ \rho_{3} \otimes yg_{20}$};

\node[draw=none,fill=none]      (21)        [below=of 11]   {$cg_{10}$};
\node[draw=none,fill=none, xshift=-.7cm]      (24)        [right=of 23]   {$yg_{6}$};
\node[draw=none,fill=none, text centered, text width=4.5em]      (25)        [right=2.55cm of 24]   {$dg_{25}+\rho_{2} \otimes yg_{30}$};
\node[draw=none,fill=none, text centered, text width=4.5em]      (26)        [right=2.1cm of 25]   {$dg_{9}+\rho_{23} \otimes y_{6}$};

\node[draw=none,fill=none]      (31)        [below=of 21]   {$yg_{12}$};
\node[draw=none, fill=none]     (32)        [right=of 31]   {};
\node[draw=none, fill=none]     (33)        [right=of 32]   {};
\node[draw=none,fill=none,text centered, text width=4.5em]      (34)        [right=of 33]   {$yg_{30}+\rho_{123}\otimes cg_{23}$};
\node[draw=none,fill=none]      (35)        [right=of 34]   {$yg_{15}$};

\node[draw=none, fill=none]     (44)        [below=of 34]   {};
\node[draw=none, fill=none]     (45)        [below=of 35]   {};
\node[draw=none, fill=none]     (36)        [below=of 26]   {};

\node[draw=none, fill=none]     (41)        [below=of 31] {};

\path[->]
(12)        edge[line width=1.5]   node            {}              (13)
(15)        edge[line width=1.5]  node            {}              (16)

(21)        edge[line width=1.5]  node            {}              (11)


(24)        edge[line width=1.5]  node            {}              (25)
;

\path[->, loosely dashed, ultra thin]
(45)        edge    node        {}      (35)
(36)        edge    node        {}      (25)
(44)        edge    node        {}      (34)
(41)        edge    node        {}      (31)
(32)        edge    node        {}      (31)
(34)        edge    node        {}      (33)
;



\path[->, draw, line width=1.5]
(14)  to [in=135, out=315]  node               {}              (26)
;

\end{tikzpicture}
\caption{Simplified portion of $\CFDh(X_{Q(4_1)})$}
\label{fig:zoomSimp}
\end{figure}

We continue this process until all unlabeled edges are pruned away from the main complex and can be homotoped away. The final result of this process is shown as the connected components in Figure \ref{fig:boxes}.

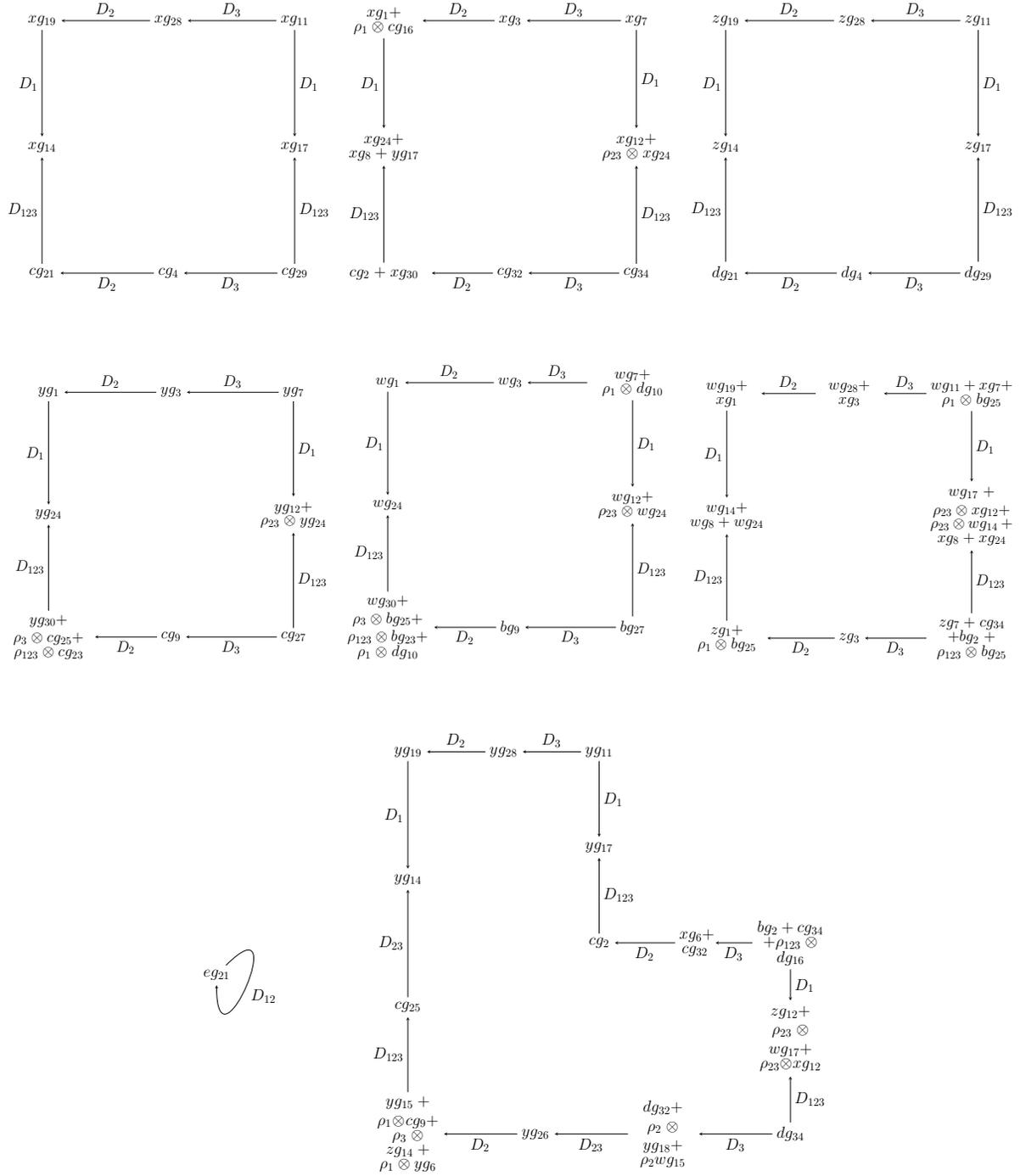
\begin{figure}[hp]
\begin{center}
\resizebox{16cm}{!}{
\begin{tikzpicture} 
[>=stealth,
   shorten >=1pt,
   node distance=4cm,
   on grid,
   auto,
   every state/.style={draw=black!60, fill=black!5, very thick}
  ]
 \node[draw=none,fill=none] (mid)                             {\Large $xg_{28}$};
 \node[draw=none,fill=none] (left)      [left=of mid]          {\Large $xg_{19}$};
 \node[draw=none,fill=none] (right)     [right=of mid]          {\Large $xg_{11}$};
 \node[draw=none,fill=none] (right2)    [below=of right]        {\Large $xg_{17}$};
 \node[draw=none,fill=none] (right3)    [below=of right2]       {\Large $cg_{29}$};
 \node[draw=none,fill=none] (mid3)      [left=of right3]        {\Large $cg_4$};
 \node[draw=none,fill=none] (left2)     [below=of left]         {\Large $xg_{14}$};
 \node[draw=none,fill=none] (left3)     [below=of left2]        {\Large $cg_{21}$};
 
\path[->]
    (mid)       edge                node[swap]          {\Large $D_2$}     (left)
    (right)     edge                node[swap]          {\Large $D_3$}       (mid)
    (right)     edge                node                {\Large $D_1$}     (right2)
    (right3)    edge                node[swap]          {\Large $D_{123}$}  (right2)
    (right3)    edge                node                {\Large $D_3$}     (mid3)
    (mid3)      edge                node                {\Large $D_2$}     (left3)
    (left3)     edge                node              {\Large $D_{123}$} (left2)
    (left)      edge                node[swap]          {\Large $D_{1}$}   (left2)
    ;
    
\end{tikzpicture}

\begin{tikzpicture} 
[>=stealth,
   shorten >=1pt,
   node distance=4cm,
   on grid,
   auto,
   every state/.style={draw=black!60, fill=black!5, very thick}
  ]
 \node[draw=none,fill=none] (mid)                             {\Large $xg_{3}$};
 \node[draw=none, fill=none,text width=5.5em, text centered] (left)      [left=of mid]  {\Large $xg_{1}+$ \\ $\rho_{1}\otimes cg_{16}$};
 \node[draw=none,fill=none] (right)     [right=of mid]          {\Large $xg_{7}$};
 \node[draw=none,fill=none, text width=7em, text centered] (right2)    [below=of right]        {\Large $xg_{12}+$ \\ $\rho_{23}\otimes xg_{24}$};
 \node[draw=none,fill=none] (right3)    [below=of right2]       {\Large $cg_{34}$};
 \node[draw=none,fill=none] (mid3)      [left=of right3]        {\Large $cg_{32}$};
 \node[draw=none,fill=none, text width=7em, text centered] (left2)     [below=of left]         {\Large $xg_{24}+$\\$xg_{8}+yg_{17}$};
 \node[draw=none,fill=none, text width=7em, text centered] (left3)     [below=of left2]        {\Large $cg_{2}+xg_{30}$};
 
\path[->]
    (mid)       edge                node[swap]          {\Large $D_2$}     (left)
    (right)     edge                node[swap]          {\Large $D_3$}       (mid)
    (right)     edge                node                {\Large $D_1$}     (right2)
    (right3)    edge                node[swap]          {\Large $D_{123}$}  (right2)
    (right3)    edge                node                {\Large $D_3$}     (mid3)
    (mid3)      edge                node                {\Large $D_2$}     (left3)
    (left3)     edge                node              {\Large $D_{123}$} (left2)
    (left)      edge                node[swap]          {\Large $D_{1}$}   (left2)
    ;
    
\end{tikzpicture}

\begin{tikzpicture} 
[>=stealth,
   shorten >=1pt,
   node distance=4cm,
   on grid,
   auto,
   every state/.style={draw=black!60, fill=black!5, very thick}
  ]
 \node[draw=none,fill=none] (mid)                             {\Large $zg_{28}$};
 \node[draw=none,fill=none] (left)      [left=of mid]          {\Large $zg_{19}$};
 \node[draw=none,fill=none] (right)     [right=of mid]          {\Large $zg_{11}$};
 \node[draw=none,fill=none] (right2)    [below=of right]        {\Large $zg_{17}$};
 \node[draw=none,fill=none] (right3)    [below=of right2]       {\Large $dg_{29}$};
 \node[draw=none,fill=none] (mid3)      [left=of right3]        {\Large $dg_4$};
 \node[draw=none,fill=none] (left2)     [below=of left]         {\Large $zg_{14}$};
 \node[draw=none,fill=none] (left3)     [below=of left2]        {\Large $dg_{21}$};
 
\path[->]
    (mid)       edge                node[swap]          {\Large $D_2$}     (left)
    (right)     edge                node[swap]          {\Large $D_3$}       (mid)
    (right)     edge                node                {\Large $D_1$}     (right2)
    (right3)    edge                node[swap]          {\Large $D_{123}$}  (right2)
    (right3)    edge                node                {\Large $D_3$}     (mid3)
    (mid3)      edge                node                {\Large $D_2$}     (left3)
    (left3)     edge                node              {\Large $D_{123}$} (left2)
    (left)      edge                node[swap]          {\Large $D_{1}$}   (left2)
    ;
    
\end{tikzpicture}

}
\vskip.4in
\resizebox{16cm}{!}{
\begin{tikzpicture} 
[>=stealth,
   shorten >=1pt,
   node distance=4cm,
   on grid,
   auto,
   every state/.style={draw=black!60, fill=black!5, very thick}
  ]
 \node[draw=none,fill=none] (mid)                             {\Large $yg_{3}$};
 \node[draw=none,fill=none] (left)      [left=of mid]          {\Large $yg_{1}$};
 \node[draw=none,fill=none] (right)     [right=of mid]          {\Large $yg_{7}$};
 \node[draw=none,fill=none,text width=7em, text centered] (right2)    [below=of right]        {\Large $yg_{12}+$\\$\rho_{23}\otimes yg_{24}$};
 \node[draw=none,fill=none] (right3)    [below=of right2]       {\Large $cg_{27}$};
 \node[draw=none,fill=none] (mid3)      [left=of right3]        {\Large $cg_9$};
 \node[draw=none,fill=none] (left2)     [below=of left]         {\Large $yg_{24}$};
 \node[draw=none,fill=none, text width=7em, text centered] (left3)     [below=of left2]        {\Large $yg_{30}+$\\$\rho_{3}\otimes cg_{25}+$\\$\rho_{123}\otimes cg_{23}$};
 
\path[->]
    (mid)       edge                node[swap]          {\Large $D_2$}     (left)
    (right)     edge                node[swap]          {\Large $D_3$}       (mid)
    (right)     edge                node                {\Large $D_1$}     (right2)
    (right3)    edge                node[swap]          {\Large $D_{123}$}  (right2)
    (right3)    edge                node                {\Large $D_3$}     (mid3)
    (mid3)      edge                node                {\Large $D_2$}     (left3)
    (left3)     edge                node              {\Large $D_{123}$} (left2)
    (left)      edge                node[swap]          {\Large $D_{1}$}   (left2)
    ;
    
\end{tikzpicture}

\begin{tikzpicture} 
[>=stealth,
   shorten >=1pt,
   node distance=4cm,
   on grid,
   auto,
   every state/.style={draw=black!60, fill=black!5, very thick}
  ]
 \node[draw=none,fill=none] (mid)                             {\Large $wg_{3}$};
 \node[draw=none,fill=none] (left)      [left=of mid]          {\Large $wg_{1}$};
 \node[draw=none,fill=none,text width=7em, text centered] (right)     [right=of mid]          {\Large $wg_{7}+$\\$\rho_{1}\otimes dg_{10}$};
 \node[draw=none,fill=none,text width=7em, text centered] (right2)    [below=of right]        {\Large $wg_{12}+$\\$\rho_{23}\otimes wg_{24}$};
 \node[draw=none,fill=none] (right3)    [below=of right2]       {\Large $bg_{27}$};
 \node[draw=none,fill=none] (mid3)      [left=of right3]        {\Large $bg_9$};
 \node[draw=none,fill=none] (left2)     [below=of left]         {\Large $wg_{24}$};
 \node[draw=none,fill=none, text width=7em, text centered] (left3)     [below=of left2]        {\Large $wg_{30}+$\\$\rho_{3}\otimes bg_{25}+$\\$\rho_{123}\otimes bg_{23}+$\\$\rho_{1} \otimes dg_{10}$};
 
\path[->]
    (mid)       edge                node[swap]          {\Large $D_2$}     (left)
    (right)     edge                node[swap]          {\Large $D_3$}       (mid)
    (right)     edge                node                {\Large $D_1$}     (right2)
    (right3)    edge                node[swap]          {\Large $D_{123}$}  (right2)
    (right3)    edge                node                {\Large $D_3$}     (mid3)
    (mid3)      edge                node                {\Large $D_2$}     (left3)
    (left3)     edge                node              {\Large $D_{123}$} (left2)
    (left)      edge                node[swap]          {\Large $D_{1}$}   (left2)
    ;
    
\end{tikzpicture}

\begin{tikzpicture} 
[>=stealth,
   shorten >=1pt,
   node distance=4cm,
   on grid,
   auto,
   every state/.style={draw=black!60, fill=black!5, very thick}
  ]
 \node[draw=none,fill=none,text width=5em, text centered] (mid)                             {\Large $wg_{28}+$\\$xg_3$};
 \node[draw=none,fill=none,text width=5em, text centered] (left)      [left=of mid]          {\Large $wg_{19}+$\\$xg_1$};
 \node[draw=none,fill=none,text width=7em, text centered] (right)     [right=of mid]          {\Large $wg_{11}+xg_{7}+$\\$\rho_{1}\otimes bg_{25}$};
 \node[draw=none,fill=none,text width=7em, text centered] (right2)    [below=of right]        {\Large $wg_{17}+\rho_{23}\otimes xg_{12}+$\\$\rho_{23}\otimes wg_{14}+xg_{8}+xg_{24}$};
 \node[draw=none,fill=none,text width=7em, text centered] (right3)    [below=of right2]       {\Large $zg_{7}+ cg_{34}$\\$+ bg_{2}+\rho_{123}\otimes bg_{25}$};
 \node[draw=none,fill=none] (mid3)      [left=of right3]        {\Large $zg_3$};
 \node[draw=none,fill=none,text width=7em, text centered] (left2)     [below=of left]           {\Large $wg_{14}+$\\$wg_{8}+wg_{24}$};
 \node[draw=none,fill=none, text width=6em, text centered] (left3)     [below=of left2]       {\Large $zg_{1}+$\\$\rho_{1}\otimes bg_{25}$};
 
\path[->]
    (mid)       edge                node[swap]          {\Large $D_2$}     (left)
    (right)     edge                node[swap]          {\Large $D_3$}       (mid)
    (right)     edge                node                {\Large $D_1$}     (right2)
    (right3)    edge                node[swap]          {\Large $D_{123}$}  (right2)
    (right3)    edge                node                {\Large $D_3$}     (mid3)
    (mid3)      edge                node                {\Large $D_2$}     (left3)
    (left3)     edge                node              {\Large $D_{123}$} (left2)
    (left)      edge                node[swap]          {\Large $D_{1}$}   (left2)
    ;
    
\end{tikzpicture}

}
\vskip.4in
\resizebox{10cm}{!}{

\begin{tikzpicture} 
[>=stealth,
   shorten >=1pt,
   node distance=3cm,
   on grid,
   auto,
   every state/.style={draw=black!60, fill=black!5, very thick}
  ]         
 \node[draw=none,fill=none] (mid)                                           {\Large $yg_{28}$};
 \node[draw=none,fill=none] (left)      [left=of mid]                       {\Large $yg_{19}$};
 <vshift
 \node[draw=none, fill=none]    (left2)     [below=4cm of left]        {\Large $yg_{14}$};
  \node[draw=none, fill=none]    (left3)     [below=4cm of left2]       {\Large $cg_{25}$};
  \node[draw=none,fill=none, text width=5em, text centered]  (left4)     [below=4cm of left3]          {\Large $yg_{15}+\rho_{1}\otimes cg_{9}+$\\$\rho_{3} \otimes zg_{14}+\rho_{1}\otimes yg_{6}$};
 \node[draw=none,fill=none] (right)     [right=of mid]                             {\Large $yg_{11}$};
 \node[draw=none,fill=none] (right2)    [below=of right]                             {\Large $yg_{17}$};
 \node[draw=none,fill=none] (right3)    [below=of right2]                         {\Large $cg_{2}$};
 \node[draw=none,fill=none, text width=2.5em, text centered] (mid2)      [right=of right3]        {\Large $xg_6+$\\$cg_{32}$};
 \node[draw=none,fill=none, text width=5.5em, text centered]  (right4)    [right=of mid2]         {\Large $bg_{2}+cg_{34}$\\+$\rho_{123}\otimes dg_{16}$};
 \node[draw=none,fill=none, text width=5em, text centered] (right5)    [below=of right4]            {\Large $zg_{12}+$\\$\rho_{23}\otimes wg_{17}+$\\$\rho_{23} \otimes xg_{12}$};
 \node[draw=none,fill=none] (right6)     [below=of right5]                                          {\Large $dg_{34}$};
 \node[draw=none,fill=none, text width=5em, text centered] (mid3)       [left=4cm of right6]        {\Large $dg_{32}+$\\$\rho_{2}\otimes yg_{18}+$\\$\rho_{2} wg_{15}$};
 \node[draw=none,fill=none]     (mid4)      [left=4cm of mid3]          {\Large $yg_{26}$};
 \node[draw=none,fill=none]     (alone)     [left=6cm of left2, yshift=-3cm]         {\Large $eg_{21}$};

\path[->]
    (mid)       edge                node[swap]          {\Large $D_2$}           (left)
    (right)     edge                node[swap]          {\Large $D_3$}            (mid)
    (right)     edge                node                {\Large $D_1$}          (right2)
    (right3)    edge                node[swap]          {\Large $D_{123}$}      (right2)
    (mid2)      edge                node                {\Large $D_2$}          (right3)
    (right4)      edge              node                {\Large $D_3$}          (mid2)
    (right4)     edge               node                 {\Large $D_{1}$}      (right5)
    (right6)      edge              node[swap]          {\Large $D_{123}$}     (right5)
    (right6)    edge                node                {\Large $D_{3}$}       (mid3)
    (mid3)      edge                node                {\Large $D_{23}$}      (mid4)
    (mid4)      edge                node                {\Large $D_{2}$}       (left4)
    (left4)     edge                node                {\Large $D_{123}$}     (left3)
    (left3)     edge                node                {\Large $D_{23}$}      (left2)
    (left)     edge                node[swap]           {\Large $D_{1}$}       (left2)
    
    ;
    
\path[->, draw]
    (alone)   to [in=-90, out=45, loop, distance=3cm]   node[auto]  {\Large ${D_{12}}$}      (alone)
    ;

\end{tikzpicture}
}
\end{center}
\caption{$\CFDh{(S^{3} \setminus \nu(Q(4_1)))}$}
\label{fig:boxes}
\end{figure}

The final step in our overall computation is to compute the knot Floer complex of the knot in $S^3$ from the bordered invariants of its complement in $S^{3}$. This process involves the following main steps:

\begin{enumerate}[label={[\Roman*]}]
    \item \label{step1} Reverse the process described in Section A.4 of \cite{Lipshitz_2018} to convert the $\delta^{1}$ map into the $\partial$ map, and to eliminate generators in the idempotent $\iota_{1}$.
    \item \label{step2} Identify which, if any, diagonal differentials we might gain when switching from the ring $\mathbb{F}[\cU, \cV]/(\cU\cV)$ to the ring $\mathbb{F}[\cU, \cV]$; and of these, which can be eliminated via a change of basis.
    \item \label{step3} Use what we know about how knot invariants such as the genus and Alexander polynomial behave under satellite operations to pin down the absolute Alexander gradings and the relative Maslov gradings of our generators.
\end{enumerate}

To begin, the reader is directed to consult A.4 of \cite{Lipshitz_2018} to complete Step \ref{step1} of this process. The result of this procedure, which gives us $\CFKUV{(Q(4_1))}$, is shown in Figure \ref{fig:HFKmodUV}.

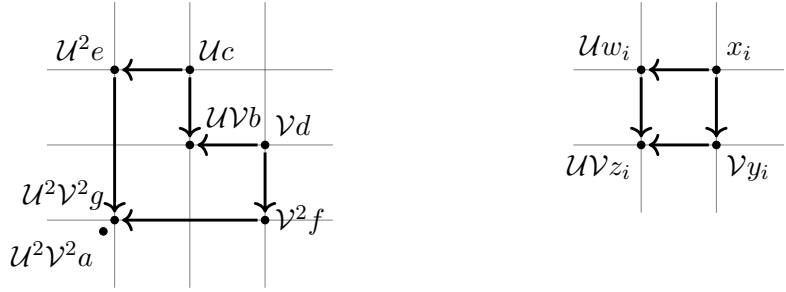
\begin{figure}
    \centering
    \begin{tikzpicture}

\draw[step=1cm,gray,very thin] (0.1,0.1) grid (3.9,3.9);

\draw [fill] (1,3) circle [radius=0.05];
\draw [fill] (2,3) circle [radius=0.05];
\draw [fill] (2,2) circle [radius=0.05];
\draw [fill] (3,2) circle [radius=0.05];
\draw [fill] (3,1) circle [radius=0.05];
\draw [fill] (1,1) circle [radius=0.05];
\draw [fill] (0.85, 0.85)   circle  [radius=0.05];

\node [above left]           at (1,3)    {$\cU^{2}e$};
\node [above right]           at (2,3)    {$\cU c$};
\node [above right]     at (2.1,2.1)    {$\cU\cV b$};
\node [above right]     at (3,2)    {$\cV d$};
\node [right]           at (3,1)    {$\cV^{2} f$};
\node [above left]     at (1,1)    {$\cU^{2}\cV^{2}g$};
\node [below left]      at (0.85,0.85)    {$\cU^{2}\cV^{2}a$};   

\draw[->, line width=1.1] (1.9,3) -- (1.1,3);
\draw[->, line width=1.1] (2,2.9) -- (2,2.1);
\draw[->, line width=1.1] (2.9,2) -- (2.1,2);
\draw[->, line width=1.1] (3,1.9) -- (3,1.1);
\draw[->, line width=1.1] (2.9,1) -- (1.1,1);
\draw[->, line width=1.1] (1,2.9) -- (1,1.1);

\draw[step=1cm,gray,very thin] (7.1,1.1) grid (9.9,3.9);

\draw [fill] (8,2) circle [radius=0.05];
\draw [fill] (9,2) circle [radius=0.05];
\draw [fill] (8,3) circle [radius=0.05];
\draw [fill] (9,3) circle [radius=0.05];

\node [below left]  at (8,2)    {$\cU \cV z_{i}$};
\node [below right]  at (9,2)    {$\cV y_{i}$};
\node [above left]  at (8,3)    {$\cU w_{i}$};
\node [above right]  at (9,3)    {$x_{i}$};

\draw[->, line width=1.1] (8.9,2) -- (8.1,2);
\draw[->, line width=1.1] (9,2.9) -- (9,2.1);
\draw[->, line width=1.1] (8.9,3) -- (8.1,3);
\draw[->, line width=1.1] (8,2.9) -- (8,2.1);

\end{tikzpicture}
    \caption{$\CFKUV{(Q(4_1))}$ where $\cR = \mathbb{F}[\cU, \cV]/(\cU\cV), i = 1$ through $6$}
    \label{fig:HFKmodUV}
\end{figure}

For Step \ref{step2}, we claim: the only diagonal differential which cannot be removed via some change of basis is the arrow corresponding to $\partial(b) = \cU\cV g$. Note that this arrow is required to ensure that $\partial^{2} =0$, and in particular that $\partial^{2}(c) = 0$ and $\partial^{2}(d) = 0$. So certainly, it must be present in the knot Floer complex. The complex at this stage is presented in Figure \ref{fig:HFKwDiag}.

\begin{figure}[h!]
    \centering
    \begin{tikzpicture}

\draw[step=1cm,gray,very thin] (0.1,0.1) grid (3.9,3.9);

\draw [fill] (1,3) circle [radius=0.05];
\draw [fill] (2,3) circle [radius=0.05];
\draw [fill] (2,2) circle [radius=0.05];
\draw [fill] (3,2) circle [radius=0.05];
\draw [fill] (3,1) circle [radius=0.05];
\draw [fill] (1,1) circle [radius=0.05];
\draw [fill] (0.85, 0.85)   circle  [radius=0.05];

\node [above left]           at (1,3)    {$\cU^{2}e$};
\node [above right]           at (2,3)    {$\cU c$};
\node [above right]     at (2.1,2.1)    {$\cU\cV b$};
\node [above right]     at (3,2)    {$\cV d$};
\node [right]           at (3,1)    {$\cV^{2} f$};
\node [above left]     at (1,1)    {$\cU^{2}\cV^{2}g$};
\node [below left]      at (0.85,0.85)    {$\cU^{2}\cV^{2}a$};   

\draw[->, line width=1.1] (1.9,3) -- (1.1,3);
\draw[->, line width=1.1] (2,2.9) -- (2,2.1);
\draw[->, line width=1.1] (2.9,2) -- (2.1,2);
\draw[->, line width=1.1] (3,1.9) -- (3,1.1);
\draw[->, line width=1.1] (2.9,1) -- (1.1,1);
\draw[->, line width=1.1] (1,2.9) -- (1,1.1);
\draw[->, line width=1.1] (1.9,1.9) -- (1.1,1.1);

\draw[step=1cm,gray,very thin] (7.1,1.1) grid (9.9,3.9);

\draw [fill] (8,2) circle [radius=0.05];
\draw [fill] (9,2) circle [radius=0.05];
\draw [fill] (8,3) circle [radius=0.05];
\draw [fill] (9,3) circle [radius=0.05];

\node [below left]  at (8,2)    {$\cU \cV z_{i}$};
\node [below right]  at (9,2)    {$\cV y_{i}$};
\node [above left]  at (8,3)    {$\cU w_{i}$};
\node [above right]  at (9,3)    {$x_{i}$};

\draw[->, line width=1.1] (8.9,2) -- (8.1,2);
\draw[->, line width=1.1] (9,2.9) -- (9,2.1);
\draw[->, line width=1.1] (8.9,3) -- (8.1,3);
\draw[->, line width=1.1] (8,2.9) -- (8,2.1);

\end{tikzpicture}
    \caption{$\CFKUV({Q(4_1}))$, where $\cR = \mathbb{F}[\cU, \cV]$}
    \label{fig:HFKwDiag}
\end{figure}

Before we consider any additional diagonal differentials, we note that any such arrows will be between two boxes or between a box and the L-shaped complex, so it will behoove us to nail down as much grading data as we can.

So, we jump to Step \ref{step3}, with the promise to circle back to finish up Step \ref{step2} afterwards. Step \ref{step3} involves pining down our complex in the $(i,j)$-plane to the extent that that is possible. We will make use of what we know about how certain knot invariants, namely, genus and the Alexander polynomial, behave under satelliting in general and the Mazur pattern in particular.  Recall that knot Floer homology detects genus, as described in Section \ref{sec:HF}. The satellite genus formula $$g(P(K)) = r\cdot g(K) + g(P \subset \mathbb{D}^{2} \times \mathbb{S}^{1})$$ where $r$ is the winding number of $P$, tells us that $g(Q(4_1)) = 2$. We can conclude from this that the Alexander gradings of our generators will be integers in the interval $[-2,2]$.

We note that as the generator of both vertical and horizontal homology in the complex, the generator $a$ has Alexander grading and Maslov grading $0$, and hence $$gr(a) = (gr_{\cU}(a),gr_{\cV}(a)) = (0,0)$$

\begin{figure}
    \centering
    \begin{tikzpicture}

\draw [stealth-stealth, line width=1.1](0,-3) -- (0,3);

\node[left] at (-.35,0) {\footnotesize $j=0$};
\node[left] at (-.35,1) {\footnotesize $j=1$};
\node[left] at (-.35,2) {\footnotesize $j=2$};
\node[left] at (-.35,-1) {\footnotesize $j=-1$};
\node[left] at (-.35,-2) {\footnotesize $j=-2$};

\draw [fill] (0,-1) circle [radius=0.08];
\draw [fill] (0,-2) circle [radius=0.08];
\draw [fill] (0,1) circle [radius=0.08];
\draw [fill] (0,2) circle [radius=0.08];
\draw [fill] (0,0) circle [radius=0.08];
\draw [fill] (-.2,0) circle [radius=0.08];
\draw [fill] (0.2,0) circle [radius=0.08];

\node [right]           at (0.2,-1)    {\large $d$};
\node [right]           at (0.2,-2)    {\large $f$};
\node [right]           at (0.2,1)    {\large $c$};
\node [right]           at (0.2,2)    {\large $e$};
\node [right]           at (0.2,2)    {\large $e$};
\node [right]           at (0.2,0)    {\large $a,g,b$};


\end{tikzpicture}
    \caption{Some generators of $\CFKh{(Q(4_1))}$}
    \label{fig:gennomaslov}
\end{figure}
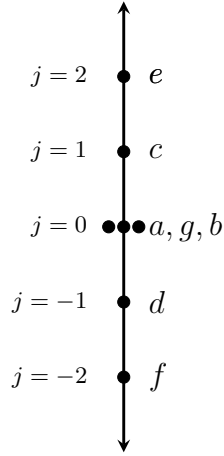

Next, $g(Q(4_1)) = 2$ tells us that our L-shaped complex must be situated in the $(i,j)$ plane precisely as shown in Figure \ref{fig:gennomaslov}. Note that, for Figures \ref{fig:gennomaslov} - \ref{fig:gen2casetwomaslov}, whenever a series of generators share a horizontal axis as in the case of $a,g,b$, these generators are still meant to be sitting along the axis $i=0$ but are staggered for visualization purposes.   

Now, it will be helpful to recall how knot Floer homology categorifies our Alexander polynomial, namely

$$t - 3 + \frac{\displaystyle 1}{\displaystyle t} = \Delta_{K}(t) = \sum_{M, s} (-1)^{M} \hspace{.1cm} t^{s} \hspace{.1cm} \text{rank}\left(\HFKh_{M}(K,s)\right)$$

\noindent where $M$ denotes the Maslov grading and $s$ the Alexander grading, and where we have used the fact that $\Delta_{Q(K)} = \Delta_{K}$ for $Q$ the untwisted Mazur pattern. Now, we note that we have 31 generators of our knot Floer complex; so, we must have that 26 of them appear in pairs $(\alpha, \beta)$ where $s(\alpha) = s(\beta)$ and $M(\alpha) \equiv M(\beta) + 1$ (mod $2$), and then the remaining 5 are arranged in the following bigradings: one generator in ($s=1$, $M \equiv m_1$ (mod 2)), three generators in ($s=0$,  $M \equiv m_{1} + 1$ (mod 2)), and one generator in bigrading ($s=-1$, $M \equiv m_1$ (mod 2)). Then, since the generators $e$ and $f$ are in Alexander gradings $s=2$ and $s=-2$ respectively, they must be canceled in the $\Delta_{K}$ computation.

Recalling the conventions for naming the generators of the square acyclic complexes as given in Figure \ref{fig:gen1maslov}, to cancel $e$ in the $\Delta_{K}$ computation we must have a square acyclic complex with generators:

\begin{itemize}
\item[*] $w_{i}$ in bigrading ($s = 2$, $M \equiv M(e) + 1 \hspace{.15cm} \text{mod} \hspace{.15cm} 2)$
\item[*] $x_{i}$ in bigrading ($s=1$, $M \equiv M(e) \hspace{.15cm} \text{mod} \hspace{.15cm} 2)$
\item[*] $y_{i}$ in bigrading ($s=0, M \equiv M(e) + 1 \hspace{.15cm} \text{mod} \hspace{.15cm} 2)$
\item[*] $z_{i}$ in bigrading ($s=1$, $M \equiv M(e) \hspace{.15cm} \text{mod} \hspace{.15cm} 2)$ 
\end{itemize}

\noindent where we've use the fact that $\partial$ lowers the Maslov grading by 1 to determine the Maslov gradings of the generators relative to one another. We will call this square complex Box 1. Similarly, to cancel $f$ in the $\Delta_{K}$ computation we must have a square acyclic complex with generators:
\begin{itemize}
\item[*] $w_{i}$ in bigrading ($s = 0$, $M \equiv M(f) + 1 \hspace{.15cm} \text{mod} \hspace{.15cm} 2)$
\item[*] $x_{i}$ in bigrading ($s=-1$, $M \equiv M(f) \hspace{.15cm} \text{mod} \hspace{.15cm} 2)$
\item[*] $y_{i}$ in bigrading ($s=-2, M \equiv M(f) + 1 \hspace{.15cm} \text{mod} \hspace{.15cm} 2)$
\item[*] $z_{i}$ in bigrading ($s=1$, $M \equiv M(f) \hspace{.15cm} \text{mod} \hspace{.15cm} 2)$ 
\end{itemize}
\noindent We denote this Box 2. From here, it is not difficult to assign relative, modulo 2 Maslov gradings for all generators mentioned thus far. We show in Figure \ref{fig:gen1maslov} the resultant complex thus far; recall that usually, Maslov gradings are suppressed from such a diagram, but we encode the relative mod 2 gradings via the following convention: a generator in relative Maslov grading $m \equiv 1$ (mod $2$) is given by an open circle and a generator in relative Maslov grading $m \equiv 0$ (mod $2$) is given by a closed circle. As in Figure \ref{fig:gennomaslov} the labels of the generators are to be read respectively to the order of the generators. For example, in Figure \ref{fig:gen1maslov}, $e$ has Alexander grading $2$ and relative Maslov grading $M \equiv 1$ mod 2, and $w_{4}$ has Alexander grading $2$ and relative Maslov grading $M \equiv 0$ mod 2.

There remain four square acyclic complexes to be placed in the plane; but note that as we have recovered the Alexander polynomial completely, these will be arranged in pairs so as to ensure all generators cancel in the computation of $\Delta_{K}$. There are two cases which permit this to happen, as shown in Figures \ref{fig:gen2caseonemaslov} and \ref{fig:gen2casetwomaslov}.

\begin{figure}
    \centering
    \begin{tikzpicture}

\draw [stealth-stealth, line width=1.1](0,-3) -- (0,3);

\node[left] at (-.5,0) {\footnotesize $j=0$};
\node[left] at (-.5,1) {\footnotesize $j=1$};
\node[left] at (-.5,2) {\footnotesize $j=2$};
\node[left] at (-.5,-1) {\footnotesize $j=-1$};
\node[left] at (-.5,-2) {\footnotesize $j=-2$};

\draw [fill=white] (-.15,2) circle [radius=0.08];
\draw [fill] (.15,2) circle [radius=0.08];

\draw [fill] (-0.2,1) circle [radius=0.08];
\draw [fill=white] (0,1) circle [radius=0.08];
\draw [fill=white] (0.2,1) circle [radius=0.08];

\draw [fill] (-.4,0) circle [radius=0.08];
\draw [fill] (-.2,0) circle [radius=0.08];
\draw [fill=white] (0,0) circle [radius=0.08];
\draw [fill] (0.2,0) circle [radius=0.08];
\draw [fill] (.4,0) circle [radius=0.08];

\draw [fill] (-0.2,-1) circle [radius=0.08];
\draw [fill=white] (0,-1) circle [radius=0.08];
\draw [fill=white] (0.2,-1) circle [radius=0.08];

\draw [fill=white] (-.15,-2) circle [radius=0.08];
\draw [fill] (.15,-2) circle [radius=0.08];

\node [right]           at (0.5,-1)    {\large $d, x_{2}, z_{2}$};
\node [right]           at (0.5,-2)    {\large $f, y_{2}$};
\node [right]           at (0.5,1)    {\large $c, x_{1}, z_{1}$};
\node [right]           at (0.5,2)    {\large $e$, $w_{1}$};
\node [right]           at (0.5,0)    {\large $a,g,b, y_{1}, w_{2}$};


\end{tikzpicture}
    \caption{Some generators of $\CFKh{(Q(4_1))}$ with relative Maslov grading data}
    \label{fig:gen1maslov}
\end{figure}
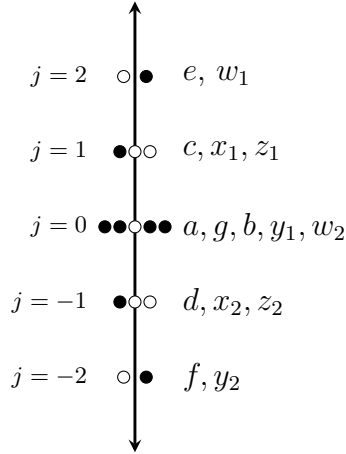

\begin{figure}
    \centering
    \begin{tikzpicture}

\draw [stealth-stealth, line width=1.1](0,-3) -- (0,3);

\node[left] at (-.6,0) {\footnotesize $j=0$};
\node[left] at (-.6,1) {\footnotesize $j=1$};
\node[left] at (-.6,2) {\footnotesize $j=2$};
\node[left] at (-.6,-1) {\footnotesize $j=-1$};
\node[left] at (-.6,-2) {\footnotesize $j=-2$};

\draw [fill] (-.15,2) circle [radius=0.08];
\draw [fill=white] (.15,2) circle [radius=0.08];

\draw [fill=white] (-0.34,1) circle [radius=0.08];
\draw [fill=white] (-.14,1) circle [radius=0.08];
\draw [fill] (0.14,1) circle [radius=0.08];
\draw [fill] (0.34,1) circle [radius=0.08];

\draw [fill] (-0.34,0) circle [radius=0.08];
\draw [fill] (-.14,0) circle [radius=0.08];
\draw [fill=white] (0.14,0) circle [radius=0.08];
\draw [fill=white] (0.34,0) circle [radius=0.08];

\draw [fill=white] (-0.34,-1) circle [radius=0.08];
\draw [fill=white] (-.14,-1) circle [radius=0.08];
\draw [fill] (0.14,-1) circle [radius=0.08];
\draw [fill] (0.34,-1) circle [radius=0.08];

\draw [fill] (-.15,-2) circle [radius=0.08];
\draw [fill=white] (.15,-2) circle [radius=0.08];

\node [right]           at (0.5,2)    {\large $w_{4}, w_{3}$};
\node [right]           at (0.5,1)    {\large $x_{4}, z_{4}, x_{3}, z_{3}$};
\node [right]           at (0.5,0)    {\large $y_{4}, w_{6}, y_{3}, w_{5}$};
\node [right]           at (0.5,-1)    {\large $x_{6}, z_{6}, x_{5}, z_{5}$};
\node [right]           at (0.5,-2)    {\large $y_{6}, y_{5}$};


\end{tikzpicture}
    \caption{Case One of Boxes 3,4,5,6 with rel. Maslov grading data}
    \label{fig:gen2caseonemaslov}
\end{figure}
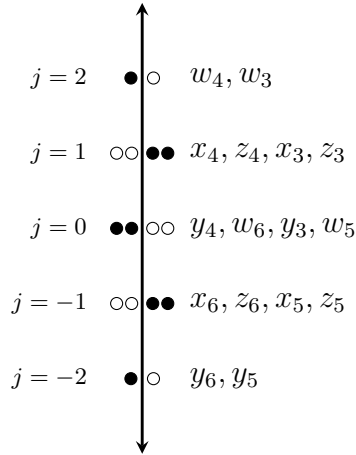

\begin{figure}
    \centering
    \begin{tikzpicture}

\draw [stealth-stealth, line width=1.1](0,-3) -- (0,3);

\node[left] at (-1.1,0) {\footnotesize $j=0$};
\node[left] at (-1.1,1) {\footnotesize $j=1$};
\node[left] at (-1.1,2) {\footnotesize $j=2$};
\node[left] at (-1.1,-1) {\footnotesize $j=-1$};
\node[left] at (-1.1,-2) {\footnotesize $j=-2$};

\draw [fill] (-0.8,1) circle [radius=0.08];
\draw [fill] (-.3,1) circle [radius=0.08];
\draw [fill=white] (0.3,1) circle [radius=0.08];
\draw [fill=white] (0.8,1) circle [radius=0.08];

\draw [fill=white] (-.9,0) circle [radius=0.08];
\draw [fill=white] (-.7,0) circle [radius=0.08];
\draw [fill=white] (-.4,0) circle [radius=0.08];
\draw [fill=white] (-.2,0) circle [radius=0.08];
\draw [fill] (0.2,0) circle [radius=0.08];
\draw [fill] (0.4,0) circle [radius=0.08];
\draw [fill] (0.7,0) circle [radius=0.08];
\draw [fill] (0.9,0) circle [radius=0.08];

\draw [fill] (-0.8,-1) circle [radius=0.08];
\draw [fill] (-.3,-1) circle [radius=0.08];
\draw [fill=white] (0.3,-1) circle [radius=0.08];
\draw [fill=white] (0.8,-1) circle [radius=0.08];

\node [right]           at (1.1,1)    {\large $w_{4}, w_{6}, w_{3}, w_{5}$};
\node [right]           at (1.1,0)    {\large $y_{4}, z_{4}, y_{6}, z_{6}, y_{3}, z_{3}, y_{5}, z_{5}$};
\node [right]           at (1.1,-1)    {\large $z_{4}, z_{6}, z_{3}, z_{5}$};


\end{tikzpicture}
    \caption{Case Two of Boxes 3,4,5,6 with rel. Maslov grading data}
    \label{fig:gen2casetwomaslov}
\end{figure}
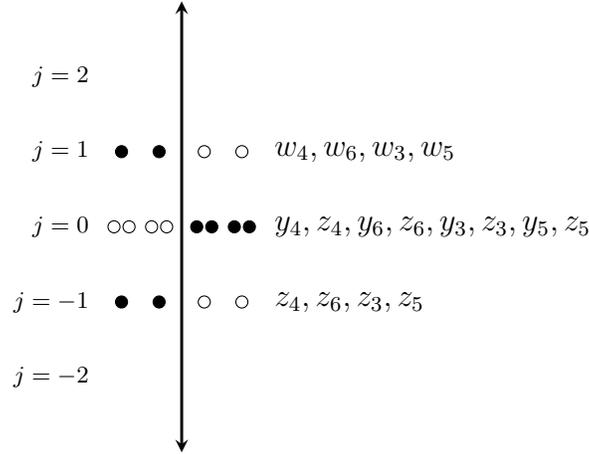

From these computations, we have enough information to determine the Alexander gradings exactly of our generators, and the mod 2 Maslov gradings. In particular, we will be able to compute $(gr_{\cU}, gr_{\cV})$ mod 2, and to compute $gr_{\cU} - gr_{\cV} = 2s$ on the nose, for every generator. This grading information will suffice for the involutive knot Floer computations in the subsequent section \ref{sec:4invol}. Note that the gradings for boxes Box 1 and Box 2 are the same in both Cases One and Two, but it will be useful for organizational purposes to repeat this data. Table \ref{grtableL} contains data common to both Cases, and Tables \ref{grtableOne} and \ref{grtableTwo} contain the grading data for Cases One and Two respectively.

\begin{table}[ht]
    \centering
    \begin{tabular}{|c|c|c|}
        \hline
        generator& $(gr_{\cU}, gr_{\cV})$ mod $2$ & $gr_{\cU} - gr_{\cV}$\\
       \hline
        a & (0,0) &     0       \\
        b & (1,1) &     0       \\
        c & (0,0) &     2       \\
        d & (0,0) &     -2      \\
        e & (1,1) &     4       \\
        f & (1,1) &     -4      \\
        g & (0,0) &     0       \\
         \hline
    \end{tabular}
   \caption{Relative gradings for the L-shaped complex}
   \label{grtableL}
\end{table}

\begin{table}[ht]
    \centering
    \begin{tabular}{|c|c|c|}
        \hline
         generator& $(gr_{\cU}, gr_{\cV})$ mod $2$ & $gr_{\cU} - gr_{\cV}$\\
         \hline
         $w_{1,4}$  &   (0,0)     &     4   \\
         $x_{1,4}$  &   (1,1)     &     2   \\
         $y_{1,4}$  &   (0,0)     &     0   \\
         $z_{1,4}$  &   (1,1)     &     2   \\
         \hline
         $w_{2,6}$  &   (0,0)     &     0    \\
         $x_{2,6}$  &   (1,1)     &     -2   \\
         $y_{2,6}$  &   (0,0)     &     -4   \\
         $z_{2,6}$  &   (1,1)     &     -2   \\
         \hline
         $w_3$      &   (1,1)   &        4      \\
         $x_3$      &   (0,0)   &        2     \\
         $y_3$      &   (1,1)   &        0      \\
         $z_3$      &   (0,0)   &        2      \\
         \hline
         $w_5$      &   (1,1)   &        0      \\
         $x_5$      &   (0,0)   &        -2     \\
         $y_5$      &   (1,1)   &        -4     \\
         $z_5$      &   (0,0)   &        -2      \\
        
         \hline
    \end{tabular}
    \caption{Relative gradings for all boxes in Case One}
    \label{grtableOne}
\end{table}

\begin{table}[ht]
    \centering
    \begin{tabular}{|c|c|c|}
        \hline
         generator& $(gr_{\cU}, gr_{\cV})$ mod $2$ & $gr_{\cU} - gr_{\cV}$\\
         \hline
         $w_1$      &   (0,0)   &        4      \\
         $x_1$      &   (1,1)   &        2     \\
         $y_1$      &   (0,0)   &        0      \\
         $z_1$      &   (1,1)   &        2      \\
         \hline
         $w_2$      &   (0,0)   &        0      \\
         $x_2$      &   (1,1)   &        -2     \\
         $y_2$      &   (0,0)   &        -4     \\
         $z_2$      &   (1,1)   &        -2      \\
         \hline
         $w_{3,5}$  &   (1,1)     &     2   \\
         $x_{3,5}$  &   (0,0)     &     0   \\
         $y_{3,5}$  &   (1,1)     &     -2   \\
         $z_{3,5}$  &   (0,0)     &     0    \\
         \hline
         $w_{4,6}$  &   (0,0)     &     2    \\
         $x_{4,6}$  &   (1,1)     &     0   \\
         $y_{4,6}$  &   (0,0)     &     -2   \\
         $z_{4,6}$  &   (1,1)     &     0   \\
        
         \hline
    \end{tabular}
    \caption{Relative gradings for all boxes in Case Two}
    \label{grtableTwo}
\end{table}

\newpage
With this data in hand, we can now consider whether any other diagonal differential(s) must be present in our complex, which cannot be eliminated via change of basis. If such differentials exist, they must preserve $\partial^{2} = 0$ and must map between generators whose Maslov gradings are opposite parity, since the differential on our complex drops this grading by 1. These requirements narrow down the possibilities for diagonal arrows to the following 5 general cases, and we will split them into 2 categories here:

\begin{enumerate}
    \item In Case A, some series of diagonal arrows exist between a box in the set {Box 1, Box 2, Box 4, Box 6} and a box in the set {Box 3, Box 5}. For convention, let us say the diagonals in this case go from some Box $\alpha$ to some Box $\beta$, and that in the $(i,j)$-plane the arrows have length $(\cU\cV)^{m}$. We note that the diagonals might have different relative lengths- in particular, in Case Two, where there are boxes along different lines of slope 1 in the $(i,j)$-plane; but the method for dealing with them is morally the same, up to perhaps additional powers of $\cV$ or $\cU$. Such cases are left to the reader to consider more closely, should they so choose. In Case A, we may have the following differential map:
     \begin{align*}
         \begin{split}
             \partial (w_{\alpha}) = & \hspace{.15cm} \cV z_{\alpha} + (\cU\cV)^{m} w_{\beta}\\
             \partial (x_{\alpha}) = & \hspace{.15cm} \cU w_{\alpha} + \cV y_{\alpha} + (\cU\cV)^{m} x_{\beta} \\
             \partial (y_{\alpha}) = & \hspace{.15cm} \cU z_{\alpha} + (\cU\cV)^{m} y_{\beta} \\
             \partial (z_{\alpha}) = & \hspace{.15cm} (\cU\cV)^{m} z_{\beta} \\
             \partial (w_{\beta}) = & \hspace{.15cm} \cV z_{\beta} \\
             \partial (x_{\beta}) = & \hspace{.15cm} \cU w_{\beta} + \cV y_{\beta} \\
             \partial (y_{\beta}) = & \hspace{.15cm}\cU z_{\beta} \\
             \partial (z_{\beta}) = & \hspace{.15cm} 0 \\
         \end{split}
     \end{align*}
     \noindent Then, to eliminate the edges between the two boxes, we perform the following changes of basis: 
         \begin{align*}
         \begin{split}
             w_{\alpha} \mapsto w_{\alpha}^{'} = & \hspace{.15cm} w_{\alpha} + \cV^{m}\cU^{m-1} x_{\beta}  \\
             z_{\alpha} \mapsto  z_{\alpha}^{'} = & \hspace{.15cm} z_{\alpha} + \cV^{m}\cU^{m-1} y_{\beta}  \\
         \end{split}
     \end{align*}
    \item For Cases B, C, D, and E, we will consider diagonal arrows that could appear between the main L-shaped complex and some box in the set {Box 1, Box 2, Box 4, Box 6}. Again, we will address mostly Case 1, with the understanding that Case 2 is identical up to a plus or minus power of $\cU$ and/or $\cV$. These 4 cases can be thought of as special cases of Case A, in that we can think of the main L complex as containing two ``sub-boxes", as outlined in the figure directly below. 
\begin{center}
    
\begin{tikzpicture}
    
\draw [fill] (1,3) circle [radius=0.05];
\draw [fill] (2,3) circle [radius=0.05];
\draw [fill] (2,2) circle [radius=0.05];
\draw [fill] (1,1) circle [radius=0.05];

\draw [fill] (2.8,1.5) circle [radius=0.05];
\draw [fill] (1.8,0.5) circle [radius=0.05];
\draw [fill] (3.8,1.5) circle [radius=0.05];
\draw [fill] (3.8,0.5) circle [radius=0.05];

\node [above left]      at (1,3)    {$\cU^{2}e$};
\node [above right]     at (2,3)    {$\cU c$};
\node [above right]     at (2.1,2.1) {$\cU\cV b$};
\node [above left]     at (1,1)    {$\cU^{2}\cV^{2}g$};

\node [above right]     at (3.8,1.5)    {$\cV d$};
\node [right]           at (3.8,0.5)    {$\cV^{2} f$};
\node [above]           at (2.8,1.5)    {$\cU\cV b$};
\node [below left]     at (1.8,0.5)    {$\cU^{2}\cV^{2}g$};

\draw[->, line width=1.1] (1.9,3) -- (1.1,3);
\draw[->, line width=1.1] (2,2.9) -- (2,2.1);
\draw[->, line width=1.1] (1,2.9) -- (1,1.1);
\draw[->, line width=1.1] (1.9,1.9) -- (1.1,1.1);

\draw[->, line width=1.1] (3.7,1.5) -- (2.9,1.5);
\draw[->, line width=1.1] (3.8,1.4) -- (3.8,0.6);
\draw[->, line width=1.1] (3.7,0.5) -- (1.9,0.5);
\draw[->, line width=1.1] (2.7,1.4) -- (1.9,0.6);

\end{tikzpicture}
\end{center}
\vspace{.5cm}

Then, the four cases are thus: 
 
 \begin{enumerate}[label={[\Alph*]}]
  \setcounter{enumii}{1}
  \item The diagonals go from some Box $\alpha$ in the set \{Box 1, Box 2, Box 4, Box 6\} to the subcomplex between the generators \{e, c, b, g\}. In such a case, we would have the following differential map:
  \begin{align*}
         \begin{split}
            \partial(w_{\alpha}) = & \hspace{.15cm} \cV z_{\alpha} + (\cU\cV)^{m}e \\
            \partial(x_{\alpha}) = & \hspace{.15cm} \cU w_{\alpha} + \cV y_{\alpha} + (\cU\cV)^{m}c \\
            \partial(y_{\alpha}) = & \hspace{.15cm} \cU z_{\alpha} + (\cU\cV)^{m}b \\
            \partial(z_{\alpha}) = & \hspace{.15cm} \cU^{m}\cV^{m+1}g \\
            \partial(e) = & \hspace{.15cm}\cV^{2}g \\
            \partial(c) = & \hspace{.15cm} \cU e + \cV b \\
            \partial(b) = & \hspace{.15cm} \cU\cV g \\
            \partial(g) = & \hspace{.15cm} 0 \\
         \end{split}
\end{align*}
     \noindent Then, to eliminate the edges between the two boxes, we perform the following changes of basis: 
         \begin{align*}
         \begin{split} 
        w_{\alpha} \mapsto  w_{\alpha}^{'} = & \hspace{.15cm} w_{\alpha} + \cU^{m}\cV^{m-1}c \\
        z_{\alpha} \mapsto z_{\alpha}^{'} = & \hspace{.15cm} z_{\alpha} + \cU^{m-1}\cV^{m}b\\
         \end{split}
     \end{align*}
     
  \item The diagonals go from the subcomplex between the generators \{e, c, b, g\} to some Box $\beta$ in the set \{Box 1, Box 2, Box 4, Box 6\}.
  
  \begin{align*}
         \begin{split}
            \partial(e) = & \hspace{.15cm}\cV^{2}g + (\cU\cV)^{m} w_{\beta} \\
            \partial(c) = & \hspace{.15cm} \cU e + \cV b + (\cU\cV)^{m} x_{\beta} \\
            \partial(b) = & \hspace{.15cm} \cU\cV g + (\cU\cV)^{m} y_{\beta} \\
            \partial(g) = & \hspace{.15cm} \cU^{m} \cV^{m-1} z_{\beta} \\
            \partial(w_{\beta}) = & \hspace{.15cm} \cV z_{\beta} \\
            \partial(x_{\beta}) = & \hspace{.15cm} \cU w_{\beta} + \cV y_{\beta} \\
            \partial(y_{\beta}) = & \hspace{.15cm} \cU z_{\beta} \\
            \partial(z_{\beta}) = & \hspace{.15cm} 0 \\
         \end{split}
\end{align*}
     \noindent Then, to eliminate the edges between the two boxes, we perform the following changes of basis: 
         \begin{align*}
         \begin{split} 
         e \mapsto e^{'} = & \hspace{.15cm} e + \cU^{m}\cV^{m-1} x_{\beta} \\
         g \mapsto g^{'} = & \hspace{.15cm} g + \cV^{m-1}\cU^{m-1}y_{\beta} \\
         \end{split}
     \end{align*}
  
  \item The diagonals go from some Box $\alpha$ in the set \{Box 1, Box 2, Box 4, Box 6\} to the subcomplex between the generators {b, d, f, g}.

  \begin{align*}
         \begin{split}
            \partial(w_{\alpha}) = & \hspace{.15cm} \cV z_{\alpha} + (\cU\cV)^{m}b \\
            \partial(x_{\alpha}) = & \hspace{.15cm} \cU w_{\alpha} + \cV y_{\alpha} + (\cU\cV)^{m}d \\
            \partial(y_{\alpha}) = & \hspace{.15cm} \cU z_{\alpha} + (\cU\cV)^{m}f \\
            \partial(z_{\alpha}) = & \hspace{.15cm} \cU^{m+1}\cV^{m}g \\
            \partial(b) = & \hspace{.15cm} \cU\cV g \\
            \partial(d) = & \hspace{.15cm} \cU b + \cV f\\
            \partial(f) = & \hspace{.15cm} \cU^{2}g \\
            \partial(g) = & \hspace{.15cm} 0 \\
         \end{split}
\end{align*}
     \noindent Then, to eliminate the edges between the two boxes, we perform the following changes of basis: 
    \begin{align*}
         \begin{split} 
         w_{\alpha} \mapsto w_{\alpha}^{'} = & \hspace{.15cm} w_{\alpha} + \cU^{m}\cV^{m-1}d \\
         z_{\alpha} \mapsto z_{\alpha}^{'} = & \hspace{.15cm} z_{\alpha} + \cU^{m-1}\cV^{m}f \\ 
         \end{split}
     \end{align*}
  
  \item The diagonals go from the subcomplex between the generators \{b, d, f, g\} to some Box $\beta$ in the set \{Box 1, Box 2, Box 4, Box 6\}.

  \begin{align*}
         \begin{split}
            \partial(b) = & \hspace{.15cm} \cU\cV g + (\cU\cV)^{m} w_{\beta}\\
            \partial(d) = & \hspace{.15cm} \cU b + \cV f +(\cU\cV)^{m} x_{\beta} \\
            \partial(f) = & \hspace{.15cm} \cU^{2}g + (\cU\cV)^{m} y_{\beta} \\
            \partial(g) = & \hspace{.15cm} (\cU\cV)^{m} z_{\beta} \\
            \partial(w_{\beta}) = & \hspace{.15cm} \cV z_{\beta} \\
            \partial(x_{\beta}) = & \hspace{.15cm} \cU w_{\beta} + \cV y_{\beta} \\
            \partial(y_{\beta}) = & \hspace{.15cm} \cU z_{\beta} \\
            \partial(z_{\beta}) = & \hspace{.15cm} 0 \\
         \end{split}
\end{align*}
     \noindent Then, to eliminate the edges between the two boxes, we perform the following changes of basis: 
         \begin{align*}
         \begin{split} 
         b \mapsto b^{'} = & \hspace{.15cm} b + \cU^{m}\cV^{m-1} x_{\beta} \\
         g \mapsto g^{'} = & \hspace{.15cm} g + \cU^{m-2}\cV^{m} y_{\beta} \\
         \end{split}
     \end{align*}
\end{enumerate}
\noindent Thus, the only diagonal arrow which must exist in the knot Floer complex is $\partial(b) = \cU\cV g$. We have then organized our basis elements sufficiently to proceed to considering the involutive knot Floer complex.
\end{enumerate}
\vspace{0.5cm}

\section{Involutive Knot Floer homology}
\label{sec:4invol}

\subsection{Involutive knot Floer homology}
The involutive knot Floer homology package of Hendricks and Manolescu makes use of a conjugation action on the knot Floer complex corresponding to reversing the orientation of the Heegaard surface, and switching the $\alpha$ and $\beta$ curves \cite{HendricksManolescu}. Let $C= \CFKUV{(K)}$ denote the knot Floer complex of $K$ over the ring $\cR = \mathbb{F}[\cU, \cV]$, with differential $\partial$. We equip our complex with a basis $\mathcal{B}$. We will make use of the notions of skew-homotopies; we will use the symbol $\simeqd$ and this will denote a homotopy that is $\mathcal{R}$-skew-equivariant and thus swaps the $\cU$ and $\cV$ actions. We introduce some formal derivatives of the differential $\partial$.

\begin{definition} Let $\eta \in \mathcal{B}$ and let $\cU^{a} \cV^{b}\eta \in C$. Then, we define:
\begin{enumerate}
    \item The \textit{$\cU$ derivative} of the element $\cU^{a} \cV^{b} \eta$ is its partial derivative with respect to the variable $\cU$, given by $D_{\cU}(\cU^{a} \cV^{b} \eta) = a \hspace{.1cm} \cU^{a-1} \cV^{b} \eta$.
    \vskip.1in
    \item The \textit{$\cV$ derivative} of the element $\cU^{a} \cV^{b} \eta$ is its partial derivative with respect to the variable $\cV$, given by $D_{\cV}(\cU^{a} \cV^{b} \eta) = b \hspace{.1cm} \cU^{a} \cV^{b-1} \eta$.
    \vskip.1in
    \item The chain map defined by the commutator $\Phi = [\partial, D_{\cU}]$ is an $\cR$ equivariant chain map. It is an endomorphism on $\CFKUV$ which raises $gr_{\cU}$ by 1 and lowers $gr_{\cV}$ by 1.
    \vskip.1in
    \item The chain map defined by the commutator $\Psi = [\partial, D_{\cV}]$ is an $\cR$ equivariant chain map. It is an endomorphism on $\CFKUV$ which lowers $gr_{\cU}$ by 1 and raises $gr_{\cV}$ by 1.
\end{enumerate}

\end{definition}

The maps $\Phi$ and $\Psi$ can be thought of as formal derivatives of the boundary map with respect to $\cU$ and $\cV$, respectively. We note that the map studied by Sarkar in \cite{Sarkar_2015} is homotopy equivalent to $Id + \Phi \circ \Psi$; we will thus refer to this as the Sarkar map, for brevity. We organize information about the image of our generators under the differential and under the Sarkar map in Tables \ref{tablesarkarL} and \ref{tablesarkarsquare}.

\begin{table}[ht]
    \centering
    \begin{tabular}{|c|c|c|c|}
        \hline
        generator & $\partial$  &   $Id+\Phi \circ \Psi$ \\
        \hline
        a   &       0       &   $a$    \\
        b   &   $\cU \cV g$   &   $b$    \\
        c   &  $\cU e+\cV b$&   $c+2\cV g = c$ \\
        d   &  $\cU b+\cV f$&   $d+\cU g$   \\
        e   &  $\cV^{2} g$   &   $e$          \\
        f   &  $\cU^{2} g$   &   $f$     \\
        g   &       0       &   $g$       \\
        \hline
    \end{tabular}
    \caption{Derivative and grading data for the L-shaped subcomplex}
    \label{tablesarkarL}
\end{table}

\begin{table}[ht]
    \centering
    \begin{tabular}{|c|c|c|c|c|}
        \hline
        generator & $\partial$  &   $Id+\Phi \circ \Psi$ \\
        \hline
        $w_i$ & $\cV z_i$ & $w_i$ \\
        $x_i$ & $\cU w_i + \cV y_i$ & $x_i + z_i$ \\
        $y_i$ & $\cU z_i$ & $y_i$ \\
        $z_i$ & 0 & $z_i$ \\
        \hline
    \end{tabular}
    \caption{Derivative data for the square subcomplexes}
    \label{tablesarkarsquare}
\end{table}

Involutive knot Floer homology assigns to the knot Floer complex $\CFKUV{(K)}$ with basis $\mathcal{B}$ and boundary map $\partial$ a special map corresponding to the aforementioned conjugation action. The specific properties this map exhibits are described in the following definition:

\begin{definition} \cite{HendricksManolescu}
Let $C$ be a chain complex with basis $\mathcal{B}$ and differential $\partial$, that is freely and finitely generated over $\mathcal{R}$. Let $\iota$ be an endomorphism on $C$. Then, $(C, \mathcal{B}, \partial, \iota)$ is called an {\em $\iota_K$-complex} if it satisfies the following properties:
\begin{enumerate}
    \item There is an isomorphism between $H_{*}(C, \partial)$ and $\mathcal{R}$ which preserves gradings, such that $1 \in \mathcal{R}$ is in grading $gr(1) = (gr_{\cU}(1), \gr_{\cV}(1)) = (0,0).$
    \item The map $\iota$ is $\mathcal{R}$-skew-equivariant, skew-graded, and skew-filtered.
    \item The map $\iota$ is not an involution precisely, but rather squares to something homotopy equivalent to the Sarkar map; $\iota^{2} \simeq Id + \Phi \circ \Psi$
\end{enumerate}

\end{definition}

\noindent Then, the conjugation action on the complex is denoted by $\iota_K$ and we refer to the package of $(\CFKUV({K}),\iota_K)$ as the $\iota_K$-complex of the knot $K$.

We will now reference ideas and notation found in \cite{Zemke_2019} and \cite{HKPS_2020} to focus our consideration of $\iota_K$-complexes. Recall that the chain homotopy type of $(\CFKUV({K}), \iota_K)$ is an invariant of the knot $K$. Let $\simeqd$ denote a skew-homotopy equivalence; that is, a $\cR$ skew-equivariant, skew-filtered chain homotopy equivalence. The symbol $\simeq$ will, as usual, denote a  $\cR$-equivariant, filtered chain homotopy equivalence. In \cite{Zemke_2019}, Zemke introduces the following definition:

\begin{definition} [\cite{Zemke_2019}, Definition 2.4] We define two $\iota_K$ complexes $(C_1, \iota_1)$ and $(C_2, \iota_2)$ to be {\em locally equivalent} if there exists chain maps 

$$F: C_1 \rightarrow C_2 \hspace{1cm} \text{and} \hspace{1cm}  G: C_2 \rightarrow C_1$$

\noindent which are filtered, bigraded, and $\mathbb{F}[\cU, \cV]$-equivariant, such that

$$ F \circ \iota_1 \simeqd \iota_2 \circ F \hspace{1cm} \text{and} \hspace{1cm} G \circ \iota_2 \simeqd \iota_2 \circ G$$

\noindent and $F$ and $G$ are isomorphisms on $H_{*}((\cU, \cV)^{-1}C_1)$ and $H_{*}((\cU, \cV)^{-1}C_2)$ respectively.

\end{definition}

We will use a modification of this notion, provided by Hom-Kang-Park-Stoffregen:

\begin{definition}[\cite{HKPS_2020}, Definition 2.7]  We define two $\iota_K$ complexes $(C_1, \iota_1)$ and $(C_2, \iota_2)$ to be {\em almost locally equivalent} if there exists chain maps

$$F: C_1 \rightarrow C_2 \hspace{1cm} \text{and} \hspace{1cm}  G: C_2 \rightarrow C_1$$

\noindent which are filtered, bigraded, and $\mathbb{F}[\cU, \cV]$-equivariant, such that

$$ F \circ \iota_1 \simeqd \iota_2 \circ F \muv \hspace{1cm}  \text{and} \hspace{1cm} G \circ \iota_2 \simeqd \iota_2 \circ G \muv $$

\noindent and $F$ and $G$ are isomorphisms on $H_{*}((\cU, \cV)^{-1}C_1)$ and $H_{*}((\cU, \cV)^{-1}C_2)$ respectively.

\end{definition}

Our goal in the remainder of this section is to provide a partial computation of the $\iota_K$-complex of $Q(4_1)$ up to almost local equivalence. Our partial computation comes in the form of a weaker notion of equivalence for $\iota_{K}$-complexes:

\begin{definition}[Defn 2.6, \cite{HKPS_2020}] 
Let $C$ be a chain complex with basis $\mathcal{B}$ and differential $\partial$, that is freely and finitely generated over $\mathcal{R}$. Then, $(C, \mathcal{B}, \partial, \iota)$ is called an {\em almost $\iota_K$-complex} if it satisfies the following properties:
\begin{enumerate}
    \item There is an isomorphism between $H_{*}(C/(\cU, \cV), \partial)$ and $\mathcal{R}/(\cU, \cV)$
    \item The map $\iota$ is a skew-graded, $\mathbb{F}$-linear endomorphism on $C/(\cU, \cV)$ 
    \item $\iota^{2} = Id + \Phi \circ \Psi \muv.$
    
\end{enumerate}

\end{definition}

Thus, we will be computing the almost $\iota_K$-complex associated to $\CFKUV{Q(4_1)}$ up to almost local equivalence. Then, we will use this data to show that no nontrivial linear combination of the almost $\iota_{K}$-complex of $Q(4_{1})$ can be locally equivalent to the almost $\iota_{K}$-complex of the uknot. By Lemma 2.9 of \cite{HKPS_2020}, we know that concordant knots must have almost locally equivalent almost $\iota_K$-complexes, and thus this result implies our knot cannot have finite order in $\cC$.

We note that our proof method is inspired by and closely follows a proof method in the paper of Hom-Kang-Park-Stoffregen. We further note that our generators labeled $\{a, b, c, d, e, f, g\}$ and the components of $\partial$ that map between them form a complex that is dual to a subcomplex of the knot Floer complex $K_n$ for $n=2$, the $(3,-1)$ cable of the figure eight knot. Specifically, it is dual to the subcomplex labeled Figure 2d in their paper, for the case $n=2$. Following the example of the authors, we will compute the almost $\iota_K$ map on $\CFKUV({Q(K)})$ up to almost local equivalence. That is, we will fix the image of the map modulo the ideal generated by $\cU$ and $\cV$. We require one more lemma before we will discuss our result. 

\begin{lemma}[Lemma 2.11, \cite{HKPS_2020}] Let f and g be graded chain maps between knot Floer complexes $(C_1, \partial_1)$ and $(C_2, \partial_2)$ over $\mathbb{F}[\cU, \cV]$, such that $\text{im} \hspace{0.1cm} \partial_1  \subset (\cU, \cV)$ and $\text{im} \hspace{0.1cm} \partial_2 \subset (\cU, \cV)$. Suppose that either $f \simeq g \muv$ or $f \simeqd g \muv$. Then,
$$f = g \muv$$
\end{lemma}
\vskip.1in

Now, we are ready to discuss our partial computation of our $\iota_K$. We provide the following lemma to summarize the relevant results of this partial computation.

\begin{lemma}
\label{lem:caseone}
When our knot Floer complex is organized as in Case One, we have that 
        \begin{enumerate}
            \item \label{lem:caseoneA} $\iota_{K}(g) = g \muv$
            \item \label{lem:caseoneB} $\iota_{K}(a) = a + g \muv$
            \item \label{lem:caseoneC} $g$ can be homotoped away from appearing in $\iota_{K}$ of any other basis elements
        \end{enumerate}
\end{lemma}

\vskip.1in
\begin{proof}
For ease of reading, we proceed claim by claim to build the proof of this lemma.
\begin{enumerate}[label={[\Alph*]}]
    \item \label{claimA} \textbf{Claim}: $$\iota_{K}(z_3) = z_5 \muv$$
                $$\iota_{K}(z_5) = z_3 \muv$$
    \textit{Proof of claim:} First, we note that all $z_i$ must satisfy $\partial \iota_{K}(z_i) = \iota_{K}(\partial z_i) = 0$. In addition, the fact that $\partial w_i = \cV z_i$ implies that $\partial \iota_{K}(w_i) = \iota_{K}(\cV z_i) = \cU \iota_{K}(z_i)$, so every element $\eta$ that appears in $\iota_{K}(z_i)$ must satisfy that $\cU \eta$ is a boundary. In particular, combined with the data from Tables \ref{grtableL} and \ref{grtableOne}, these facts give us that $\iota_{K}(z_3)$ is an $\mathbb{F}[\cU\cV]$-linear combination of $z_5$ and $\cU^{2}z_3$, and $\iota_{K}(z_5)$ is an $\mathbb{F}[\cU\cV]$-linear combination of $z_3$ and $\cV^{2}z_3$. Our claim follows immediately.
    
    \vspace{0.25cm}
    \item \label{claimB} \textbf{Claim}: 
     $$\iota_{K}(g) = g \muv$$
    \textit{Proof of claim:} We note that $\partial \iota_{K}(g) = \iota_{K}(\partial g) = 0$, and that $g$ is in grading $(0,0)$ mod 2, with $gr_{\cU} - gr_{\cV} = 0$. So, we know that $\iota_{K}(g) \muv$ will be a linear combination of $\{a, g\}$. But also, as $\partial \iota_{K}(b) = \iota_{K}(\partial b) = \iota_{K}(\cU\cV g) = \cV\cU\iota_{K}(g)$, so then any element $\mu$ that appears in $\iota_{K}(g)$ must satisfy that $\cV\cU \mu$ is a boundary. Thus, $a$ is not in $\iota_{K}(g)$, and we conclude $\iota_{K}(g) = g \muv$.
    \vspace{0.25cm}
    \item \label{ClaimC} \textbf{Claim}: $$\iota_{K}(a) = a + g \muv$$ 
    This statement is proved via the proof of the following claims:
    \vspace{0.5cm}
        \begin{enumerate}[label={[\alph*]}]
        \item \textbf{Claim}: $$\iota_{K}(a) \muv \hspace{.1cm} \text{is a} \hspace{.1cm} \text{linear combination of} \hspace{.1cm} \{a, g\}$$
        \textit{Proof of claim:} This follows from identical logic as in the proof of \textbf{Claim} \ref{claimB}.
        
        \vspace{0.3cm}
        \item \textbf{Claim}: $$\iota_{K}(a) \neq 0 \muv$$
        \textit{Proof of claim:} Since $\iota_{K}$ squares to something homotopic to the Sarkar map, then if for sake of contradiction $\iota_{K}(a) = 0 \muv$, we have that
        $$0 = \iota_{K}^{2}(a) = a + H(\partial a) + \partial H(a) = a + \partial H(a)$$ \noindent But then this implies that $a = \partial H(a)$, and $a$ is not a boundary, which gives us our contradiction. 
        \vspace{0.3cm}
        \item \textbf{Claim}: $$\iota_{K}(a) \neq g \muv$$
        \textit{Proof of claim:} For sake of contradiction, if $\iota_{K}(a) = g \muv$, then combined with what we showed in \textbf{Claim} \ref{claimB}, we have that $\iota_{K}^{2}(a) = \iota_{K}(g) = g \muv$. However, we also have that $\iota_{K}^{2}(a) = a + H(\partial a) + \partial H(a) = a + \partial H(a)$. These statements combined would imply that $$g = a + \partial H(a) \muv$$
        $$\therefore \hspace{.25cm} a + g = \partial H(a) \muv$$
        \noindent which is a contradiction, as neither $a$ nor $g$ are boundaries. 
        \vspace{0.3cm}
        \item \textbf{Claim}: $$\iota_{K}(a) \neq a \muv$$ 
            \begin{enumerate}[label={[\Roman*]}]
                \item For sake of contradiction we assume that $\iota_{K}(a) = a \muv$.
                \vspace{0.25cm}
                \item Because $\iota_{K}^{2} \simeq 1 + \Phi\Psi$, we know that the map $H$ realizing this homotopy must satisfy:  
                
                \begin{equation} \label{eq1}
                \begin{split}
                \iota_{K}^{2}(d) & = d + \cU g + \partial H(d) + H(\partial d) \\
                & = d + \cU g + \partial H(d) + \cU H(b) + \cV H(e)
                \end{split}
                \end{equation}
                and, simultaneously, 
                \begin{equation} \label{eq2}
                \begin{split}
                \iota_{K}^{2}(c) & = c + \partial H(c) + H(\partial c) \\
                & = c + \partial H(c) + \cV H(b) + \cU H(f)
                \end{split}
                \end{equation}
                We will consider two subcases:
                \begin{enumerate}[label={[\roman*]}]
                    \item Either $g$ appears in $H(b)$; in which case, the \hspace{.01cm} $\cU g$ on the right side of equation \ref{eq1} is canceled, and thus the left side of equation \ref{eq1} admits $0$ (mod $2$) copies of \hspace{.01cm} $\cU g$. Also, this produces a copy of $\cV g$ on the right side of equation \ref{eq2}, and thus the left side of equation \ref{eq2} must admit $1$ (mod $2$) copies of $\cV g$.
                    \item Or, $g$ does not appear in $H(b)$; in which case, the copy of \hspace{.01cm} $\cU g$ on the right side of equation \ref{eq1} is not canceled, and thus the left side of equation \ref{eq1} must admit $1$ (mod $2$) copies of \hspace{.01cm} $\cU g$. Also, this means that the right side of equation \ref{eq2} and likewise, its left side as well, have 0 (mod 2) copies of \hspace{.01cm} $\cV g$.
                    \vspace{.25cm}
                \end{enumerate}
                
                \item \label{ClaimIII} The above two subcases indicate that, regardless of which case we are in, \\
                $$\hspace{1in} (\text{copies of} \hspace{.15cm} \cV g \hspace{.15cm} \text{in} \hspace{.15cm} \iota_{K}^{2}(c)) - (\text{copies of} \hspace{.15cm} \cU g \hspace{.15cm} \text{in} \hspace{.15cm} \iota_{K}^{2}(d)) \equiv 1 \hspace{.15cm} \text{mod} \hspace{.15cm} 2$$
                \vspace{.05cm}
                
                However, we claim that if $\iota_{K}(a)  = a \muv$, that in fact these amounts must be congruent mod 2, which gives us our contradiction.
                \vspace{.25cm}
                \\
                The proof of this claim involves identifying which elements $\vartheta$ can appear in $\iota_{K}(d)$ and in turn contribute a factor of $\cU g$ to $\iota_{K}^{2}(d)$; that is, for which $\cU g$ appears in $\iota_{K}(\vartheta)$. (Likewise, we will identify which elements $\vartheta^{'}$ can appear in $\iota_{K}(c)$ and for which $\cV g$ appears in $\iota_{K}(\vartheta^{'})$.) 
                
                \begin{enumerate}[label={[\roman*]}]
                    \item We begin with the assumption in statement [I]; in particular, that $g$ does not appear in the image of $\iota_{K}(a) \muv$. 
                    \item The only basis elements $\vartheta$ which could map to $\cU g$, based on the grading information given in Tables \ref{grtableL} and \ref{grtableOne}, are $\mathbb{F}[\cU \cV]$-factors of $$\hspace{4cm} \{ \cV g, c, \cV^{2} d, \cU w_{1}, \cU w_{4}, \cV y_{1}, \cV y_{4}, \cV w_{2}, \cV w_{6}, \cV^{3} y_{2}, \cV^{3} y_{6}, x_{3}, z_{3}, \cV^{2} x_{5}, \cV^{2} z_{5} \}$$
                    \noindent Note we have removed from consideration $\cV a$ per our assumption in part [i].
                    \item \label{b2b} As a chain map, $\iota_{K}$ must take boundaries to boundaries. Since $\cU g$ is not a boundary, it cannot appear in the image of any boundary under this map. Thus, $\cU g$ cannot appear in the image of any of the following elements: $$\hspace{4cm} \{  \cU w_{1}, \cU w_{4}, \cV y_{1}, \cV y_{4}, \cV^{3} y_{2}, \cV^{3} y_{6}, \cV^{2} z_{5} \}$$
                        \item \label{b2b2} For morally the same reasoning as used in [iii], $\cU g$ can appear in neither $\iota_{K}$ of either $\cV w_{2}$ nor $\cV w_{6}$. While these elements are not boundaries over $\mathbb{F}[\cU, \cV]$, a factor of $g$ appearing in their image under $\iota_{K}$ will still violate how the map must commute the differential. Consider: if $\cU g$ were to appear in $\iota_{K}(\cV w_{i})$, this would imply that $g$ appears in $\iota_{K}(w_{i})$, which in turn tells us that as $$\hspace{1.3in} \partial \iota_{K}(x_{i}) = \iota_{K}(\partial x_{i}) = \iota_{K}(\cU w_{i} + \cV y_{i}) = \cV \iota_{K}(w_{i}) + \cU \iota_{K}(y_{i})$$ then $\cV g$ must be a boundary, which it is not. 
                    \item \label{biglist} We now have that any of the elements $\{ c, \cV g, \cV^{2} d, x_{3}, \cV^{2} x_{5}, z_{3} \}$ can appear in $\iota_{K}(d)$ and could contribute a factor of $\cU g$ to $\iota_{K}^{2}(d)$; identical reasoning with skew-graded basis elements gives us that the only elements that can appear in $\iota_{K}(c)$ and contribute a factor of $\cV g$ to $\iota_{K}^{2}(c)$ are $\{ d, \cU g, \cU^{2} c, x_{5}, \cU^{2} x_{3}, z_{5} \}$. We further claim that in fact:
                    \begin{itemize}[leftmargin=1.5cm]
                        \item $\cV^{2} d$ is in $\iota_{K}(d)$ if and only if $\cU^{2}c$ is in $\iota_{K}(c)$ and $g$ is in $\iota_{K}(\cU d) $ if and only if $ g$ is in $\iota_{K}(\cV c)$
                        \item $\cV^{2} x_{5}$ is in $\iota_{K}(d) $ if and only if $ \cU^{2} x_{3}$ is in $\iota_{K}(c)$, and $g$ is in $\iota_{K}(\cU x_{5}) $ if and only if $ g$ is in $\iota_{K}(\cV x_{3})$
                        \item $x_{3}$ is in $\iota_{K}(d) $ if and only if $ x_{5}$ is in $\iota_{K}(c)$, and $\cU g$ is in $\iota_{K}(x_{3}) $ if and only if $ \cV g$ is in $\iota_{K}(x_{5})$
                        \item $z_{3}$ is in $\iota_{K}(d) $ if and only if $ z_{5}$ is in $\iota_{K}(c)$, and $\cU g$ is in $\iota_{K}(z_{3}) $ if and only if $ \cV g$ is in $\iota_{K}(z_{5})$
                    \end{itemize}
                    \item We will prove the first claim; the subsequent three claims follow precisely the same proof structure and logic. The crux of this statement is that $c$ must appear in $\iota_{K}(d)$ and $d$ must appear in $\iota_{K}(c)$, and that they are the only basis elements that do so; namely, $d$ is the only basis element that appears in $\iota_{K}(c)$ and vice versa. This fact is easy to see from studying the grading Tables \ref{grtableL} and \ref{grtableOne}, and by the fact that the Sarkar map reduces to the identity$\muv$ on all our basis elements. Extending this argument tells us that also $e$ must appear in $\iota_{K}(f)$ and vice versa, and $b$ must appear in $\iota_{K}(b)$. We are now ready to outline the arguments of our claim. ({\em A style note: we here move outside of our nested list for clarity of argument and to avoid a too-deeply-nested list. We return to it below}).

\end{enumerate}
\end{enumerate}
\end{enumerate}
\end{enumerate}
\vspace{.5cm}

\begin{proof} Suppose that $\cV^{2}d$ is in $\iota_{K}(d)$. Since $\iota_{K}$ commutes the differential, $\partial \cV^{2} d = \cV^{3} f + \cV^{2} \cU b$ appears in $\partial \iota_{K}(d) = \cV \iota_{K}(b) + \cU \iota_{K}(f)$. This tells us that $\cV^{2}f$ must be in $\iota_{K}(b)$. Since also $\iota_{K}(\partial c) = \partial \iota_{K}(c)$, we must have that either $\cU \cV d$ is in $\iota_{K}(c)$, or $\cU\cV f$ is in $\iota_{K}(e)$.
\begin{enumerate}[leftmargin=1.5cm, labelsep=.1cm,align=left]

\item \label{part1} Suppose $\cU \cV d$ is in $\iota_{K}(c)$. Then, the rest of $\partial \cU \cV d$ must also appear in $\partial \iota_{K}(c)$; namely, either $\cU^{2}b$ appears in $\iota_{K}(e)$ or $\cU \cV b$ appears in $\iota_{K}(b)$.

\begin{enumerate}[leftmargin=1.5cm, labelsep=.1cm,align=left, label=(1.\alph*)]

\item \label{part1b} Suppose $\cU^{2 }b$ appears in $\iota_{K}(e)$. Then, since $\partial \iota_{K}(e) = \iota_{K}(\partial e)$, we have that $\cU \cV g$ must appear in $\iota_{K}(g)$. This in turn tells us that $\cV^{3} \cU g$ appears in $\partial \iota_{K}(f)$, so either:
    \begin{enumerate}[leftmargin=1.5cm, labelsep=.1cm,align=left, label=(1.a.\roman*)]
        \item \label{part1bi} $\cU \cV e$ appears in $\iota_{K}(f)$; then, we have that $\cU^{2} \cV e$ appears in $\partial \iota_{K}(d)$. We have two possible cases:
            \begin{enumerate}[leftmargin=1.5cm, labelsep=.1cm,align=left, label=(1.a.i.\Alph*)]
            \item \label{part1aiA} $\cU^{2} e$ appears in $\iota_{K}(b)$; this in turn would imply that $\cU^{2} c$ must appear in $\iota_{K}(c)$, our desired result.
            \item \label{part1aiB} $\cU \cV c$ appears in $\iota_{K}(d)$; but this would mean that the factor of $\cU \cV^{2} b$ must cancel in $\partial \iota_{K}(d)$, and thus since $\cV^{2} b$ is in $\iota_{K}(f)$, then also $\cU \cV b$ appears in $\iota_{K}(b)$. But then we have $\cV^{2} f$ and $\cU \cV b$ are both in $\iota_{K}(b)$, so too must $\cU^{2} e$ appear in $\iota_{K}(b)$. But this is the statement of \ref{part1aiA} and hence a contradiction.
            \end{enumerate}
            
        \item \label{part1bii} $\cV^{2} b$ appears in $\iota_{K}(f)$. Then, we have the following facts: $\cV^{2} f$ appears in $\iota_{K}(b)$ and $\cV^{2} b$ appears in $\iota_{K}(f)$. So, we look to the Sarkar map; we have that
        \begin{align*}
            \begin{split}
            \iota_{K}^{2}(b) & = b + \partial H(g) + H (\partial b) \\
            & = \iota_{K}(b + \cV^{2} f + \cdots) \\
            & = b + \cU^{2} e + \cU^{2} \cV^{2} b  + \cdots
            \end{split}
        \end{align*}
        \noindent Then, there exists some element $\alpha$ in $\iota_{K}(b)$ such that $\cU^{2} \cV^{2} b$ is in $\iota_{K}(\alpha)$. Based on grading data alone, such an $\alpha$ must be in the set 
        $$\hspace{2in} \{ \cU^{2} e, \cU \cV b, \cU^{2} w_{1}, \cU^{2} w_{4}, y_{1}, y_{4}, w_{2}, w_{6}, \cV^{2} y_{2}, \cV^{2} y_{6}, \cU x_{3}, \cU z_{3}, \cV x_{5}, \cV z_{5} \} $$
        
        \noindent By hypothesis, $\cU^{2} e$ and $\cU \cV b$ are not in $\iota_{K}(b)$. Then, we can use an $\muv$ argument to eliminate any basis elements, since we know that the only basis element that appears in $\iota_{K}(b)$ must have any basis elements in its image appear in $\iota_{K}^{2}(b) \muv$; in particular, only $b$. Then, $\cU x_{3}$ is not in $\iota_{K}(b)$, as else $\cU^{2} x_{3}$ would appear in $\partial \iota_{K}(c)$, and it is not a boundary. Likewise, $\cV x_{5}$ cannot appear in $\iota_{K}(b)$ since $\cV^{2} x_{5}$ would appear in $\partial iota_{K}(d)$ and it is not a boundary. If $\cU^{2} \cV^{2} b$ were to appear in $\iota_{K}(\cU z_{3})$, then this implies that $\cU^{2}\cV b$ appears in $\iota_{K}(z_{3})$. This in turn tells us that $\cU^{3} \cV^{2} g$ appears in $\iota_{K}(z_{3})$; it must cancel, since $\partial z_{3} = 0$. But the only two options are $\cU^{3} f$ in $\iota_{K}(z_{3}$, which it cannot, since $\partial \iota_{K}(w_{3}) = \cU \iota_{K}(z_{3})$ and $\cU^{4} f$ is not a boundary; or $\cU \cV ^{2}$ is in $\iota_{K}(z_{3})$; but $\cU^{2} \cV b + \cU \cV^{2} e$ is not actually a boundary over $\mathbb{F}[\cU, \cV]$. Finally, we can use the Sarkar map $\muv$ to conclude that up to a change of basis, $y_{2}$ is in $\iota_{K}(w_{1})$ and vice versa, and they are the unique basis elements in each other's images. This tells us that should $\cV^{2} y_{2}$ appear in $\iota_{K}(b)$ with $\cU^{2} \cV^{2} b$ in its image under $\iota_{K}$, then so too must $\cU^{2} w_{1}$ appear in $\iota_{K}(b)$ with $\cU^{2} \cV^{2} b$ in its image under $\iota_{K}$; and thus these factors of $\cU^{2} \cV^{2} b$ would cancel in $\iota_{K}^{2}(b)$. A similar argument tells us that $\cU^{2} w_{4}$ and $\cV^{2} y_{6}$ will behave in the same way. Thus, in accordance with our hypotheses, there is no element that can appear in $\iota_{K}(b)$ and contribute a $\cU^{2} \cV^{2} b$ to $\iota_{K}^{2}(b)$. 
    \end{enumerate}

\item \label{part1a} $\cU \cV b$ appears in $\iota_{K}(b)$. Then, since $\partial \iota_{K}(b) = \iota_{K}(\partial b)$, we can conclude that $\cU \cV g$ appears in $\iota_{K}(g)$. Since $\partial \iota_{K}(e) = \cU^{2} \iota_{K}(g)$, then either:
    \begin{enumerate}[leftmargin=1.5cm, labelsep=.1cm,align=left, label=(1.b.\roman*)]
        \item \label{part1ai} $\cU \cV f$ appears in $\iota_{K}(e)$; but this is the statement of \ref{part2}, and hence a contradiction.
        \item \label{part1aii} $\cU^{2} b$ appears in $\iota_{K}(e)$; but this is the statement of \ref{part1b}, and hence a contradiction.
    \end{enumerate}
\end{enumerate}

\item \label{part2} Suppose $\cU\cV f$ is in $\iota_{K}(e)$.  Then, we can use this to again show that $\cU \cV g$ must appear in $\iota_{K}(g)$, and thus $\cV^{3} \cU g$ appears in $\partial \iota_{K}(f)$. So we have that either:
\begin{enumerate}[leftmargin=1.5cm, labelsep=.1cm,align=left, label=(2.\alph*)]
    \item \label{part2a} $\cU \cV e$ appears in $\iota_{K}(f)$. Then, this tells us that $\cU^{2} \cV e$ must appear in $\iota_{K}(d)$, so either: 
        \begin{enumerate}[leftmargin=1.5cm, labelsep=.1cm,align=left, label=(2.a.\roman*)]
            \item \label{part2ai} $\cU^{2} e$ appears in $\iota_{K}(b)$. Then, since $\partial \iota_{K}(b) = \iota_{K}(\partial b)$, either $\cU \cV b$ and $\cV^{2} f$ are both in $\iota_{K}(b)$, or neither are (in either case, we will have 1 mod 2 copies of $\iota_{K}(\partial b)$ in $\partial \iota_{K} (b)$, so these outcomes are logically equivalent). Since we have assumed that $\cV^{2} f$ is in $\iota_{K}(b)$ in our originally hypothesis, then $\cU \cV b$ must be in $\iota_{K}(b)$. But this would imply that $\cU \cV d$ appears in $\iota_{K}(c)$, which is the statement of \ref{part1} and hence a contradiction.  
            \item \label{part2aii} $\cU \cV c$ appears in $\iota_{K}(d)$. This would mean that there are two copies of $\cV^{2} \cU b$ in $\partial \iota_{K}(d)$ (one each from $\cV^{2} d$ and $\cU \cV c$) and hence must cancel; namely, $\cV^{2} b$ must appear in $\iota_{K}(f)$ and $\cU \cV b$ must appear in $\iota_{K}(b)$. But, then as in \ref{part2ai} we must have 1 mod 2 copies of $\iota_{K}(\partial b)$ in $\partial \iota_{K} (b)$. So then we must have also $\cU^{2} e$ appears in $\iota_{K}(b)$; but this is the statement of \ref{part2ai}, and hence a contradiction.
        \end{enumerate}
    \item \label{part2b} $\cV^{2} b$ appears in $\iota_{K}(f)$. Then, we have the same hypotheses as in \ref{part1bii} and so can repeat the proof process there to arrive at a contradiction.
\end{enumerate}
\end{enumerate}
\end{proof}
\vspace{1cm}
 
\begin{enumerate}[leftmargin=1.65in, label={[\roman*]}]
\setcounter{enumi}{7}
\item The reverse implication is exactly the same argument with skew-gradings; and thus we can conclude that $\cV^{2} d$ appears in $\iota_{K}(d) \iff \cU^{2} d$ appears in $\iota_{K}(c)$. The argument that $g$ appears in $\iota_{K}(\cU d) \iff g$ appears in $\iota_{K}(\cV c)$ is identical to an argument we have used many other times in this paper; namely, that $c$ is the unique basis element in the image of $\iota_{K}(d)$, $d$ is the unique basis element in the image of $\iota_{K}(c)$, and $g$ is the unique basis element in the image of $\iota_{K}(g)$. Tracing through the Sarkar map with these facts in mind will give us our desired result. 
\item Then, we assume that we have shown all four bulleted statements listed in \ref{biglist}; it is thus a linear algebra argument to conclude that regardless of which the images of $\iota_{K}(d)$ and $\iota_{K}(c)$, and regardless as to which elements do in fact contribute the requisite $\cU g$ (resp. $\cV g$) factors, we will always have that the copies of $\cU g$ in $\iota_{K}^{2}(d)$ and the copies of $\cV g$ in $\iota_{K}^{2}(c)$ are congruent mod 2, which contradicts \ref{ClaimIII}. Thus, $\cV a$ must both appear in $\iota_{K}(d)$ and, more significantly, must contribute a copy of $\cU g$ to $\iota_{K}^{2}(d)$. Indeed, then $a$ and $g$ both must appear in the image of $\iota_{K}(a)$. 

\end{enumerate}

\vspace{.5cm}

\begin{enumerate}[label={[\Alph*]}]
\setcounter{enumi}{3}
\item  It remains to show that, should $g$ appear in the image of $\iota_{K}$ of any basis elements besides $g$ and $a$, it can be homotoped away. Indeed, we claim the only cases we must consider for Case One are if $g$ appears in $\{\cU c, \cV d, \cU x_{3}, \cV x_{5} \}$; this follows from the work done in Claims \ref{b2b} and \ref{b2b2}.
\begin{enumerate}
\item \label{x3skewhom} First, we consider if $\cU g$ appears in $\iota_{K}(x_{3})$. Then, we can consider $\iota_{K}$ and $\iota_{K}^{'}$ such that $\iota_{K} + \iota_{K}^{'}$ is zero on every basis element except for $x_{3}$, where $(\iota_{K} + \iota_{K}^{'})(x_{3}) = \cU g$. Then, we claim there exists a skew-homotopy $H$ such that $\iota_{K} + \iota_{K}^{'} = \partial H + H \partial$. Indeed, the skew-homotopy defined by $H(y_{3}) = g$ and is zero on all other basis elements can easily be seen to satisfy the requirements. Similarly, following this exact argument, we can homotope away the following multiples of $g$ from the remaining basis elements via the given skew-homotopies: 
    \begin{enumerate}
        \item \label{x5skewhom} $\cV g$ can be homotoped away from $\iota_{K}(x_{5})$ via the skew-homotopy defined by $H(w_{5}) = g$, $H(*) = 0$ for $*$ all other basis elements
        \item \label{cskewhom} Simultaneously we can homotope away $\cU g$ from $\iota_{K}(c)$ and $\cV g$ from $\iota_{K}(d)$ via the skew-homotopy defined by $H(b) = g$, $H(*) = 0$ for $*$ all other basis elements
    \end{enumerate}
\end{enumerate}
\end{enumerate}

\end{proof}

Case Two requires an argument which is structurally the same as that of Case One, with a few caveats. It is perhaps easiest to first present the lemma statement, and identify where in the proof of the Lemma \ref{lem:caseone} the analogous proof will diverge. We will make use of the nested list labels above to reference specific steps of the proof.

\begin{lemma}
\label{lem:casetwo}
When our knot Floer complex is organized as in Case Two, we have that:
        \begin{enumerate}
            \item \label{lem:casetwoA}$\iota_{K}(g) = g \muv$
            \textbf{or}
            $\iota_{K}(g) = g + z_{3} + z_{5} \muv$
            \item \label{lem:casetwoB} $\iota_{K}(a) = a + g \muv$
            \textbf{or}
            $\iota_{K}(a) = a + g + z_{3} + z_{5}$
            \item \label{lem:casetwoC} $g$ can be removed via a change of basis from appearing in $\iota_{K}$ of any other basis elements
        \end{enumerate}
\end{lemma}
\begin{proof} Here we proceed claim by claim from the proof of Lemma \ref{lem:caseone}.
\begin{enumerate}[label={[\Alph*]}]
    \item \label{claimA2} \textbf{Claim}: $$\iota_{K}(z_3) = z_5 \muv$$
                $$\iota_{K}(z_5) = z_3 \muv$$
                OR
                $$\iota_{K}(z_3) = z_5 + g \muv$$
                $$\iota_{K}(z_5) = z_3 + g \muv$$
    \textit{Proof of claim:} Since $z_{3}$ and $z_{5}$ are again in different relative gradings from the other $z_{i}$, it follows by identical logic to Lemma \ref{lem:caseone}, \textbf{Claim} \ref{claimA}, that $\iota_{K}(z_{3})$ is an $\mathbb{F}[\cU \cV]$-linear combination of $z_{3}$ and $z_{5}$ and likewise $\iota_{K}(z_{5})$ is an $\mathbb{F}[\cU\cV]$-linear combination of $z_{5}$ and $z_{3}$. Then, we have the following options for their images, according to the Sarkar map:
       \begin{enumerate}
            \item   \begin{align*}
                    \begin{split}
                    \iota_{K}(z_{3}) = & z_{3} \muv \\
                    \iota_{K}(z_{5}) = & z_{5} \muv
                    \end{split}
                    \end{align*}
            \item   \begin{align*}
                    \begin{split}
                    \iota_{K}(z_{3}) = & z_{3} + z_{5} \muv \\
                    \iota_{K}(z_{5}) = & z_{5} \muv
                    \end{split}
                    \end{align*}
            \item   \begin{align*}
                    \begin{split}
                    \iota_{K}(z_{3}) = & z_{3} \muv \\
                    \iota_{K}(z_{5}) = & z_{5} + z_{3} \muv
                    \end{split}
                    \end{align*}
            \item \label{zsmap}   \begin{align*}
                    \begin{split}
                    \iota_{K}(z_{3}) = & z_{5} \muv \\
                    \iota_{K}(z_{5}) = & z_{3} \muv
                    \end{split}
                    \end{align*}
        \end{enumerate}
    \noindent All of these choices are equivalent up to some appropriate change of basis; so we can conclude that \ref{zsmap} is our map, without loss of generality. We note that in Case Two, we may also have that $g$ appears in the image of $\iota_{K}(z_{3})$ or $\iota_{K}(z_{5})$; the argument above demonstrates how $g$ appears in the image of $\iota_{K}(z_{3})$ if, and only if, it appears in the image of $\iota_{K}(z_{5})$. Thus our claim is proved.

 \item \label{claimA3} \textbf{Claim}: $$\iota_{K}(g) = g \muv$$
               OR
			 $$\iota_{K}(g) = g + z_{3} + z_{5} \muv$$
	We note that this is the first significant place that Lemma \ref{lem:caseone} and Lemma \ref{lem:casetwo} differ. In Case Two, we have that $z_{3}$ and $z_{5}$ have the same Alexander grading as $g$, and so we cannot definitively say that they do not appear in $\iota_{K}(g) \muv$, but we can say that $z_{3}$ appears in $\iota_{K}(g) \iff z_{5}$ does. Note that the proof that $\iota_{K}(g)$ is an $\mathbb{F}[\cU \cV]$-linear combination of $\{g, z_{3}, z_{5} \}$ is identical to the proof in Lemma \ref{lem:caseone}, Claim \ref{claimB} that $\iota_{K}(g)$ is an $\mathbb{F}[\cU \cV]$-linear combination of $\{g, \cU z_{3}, \cV z_{5} \}$. From there, we suppose for sake of contradiction that $\iota_{K}(g) = g + z_{3} \muv$.  Then, \textbf{Claim} \ref{claimA} combined with the homotopy equivalence to the Sarkar map implies that $\iota_{K}^{2}(g) = g + z_{3} + z_{5} = g + H(\partial g) + \partial H(g) = g + \partial H(g)$; but this cannot be true since neither $z_{3}$ nor $z_{5}$ are boundaries. It is logically identical to show that $\iota_{K}(g) \neq g + z_{5} \muv$. Finally, the fact that $g$ appears in neither $\iota_{K}(z_{3})$ nor $\iota_{K}(z_{5})$ indicates that $g$ must appear in $\iota_{K}(g)$, in order to preserve the homotopy equivalence to the Sarkar map. Thus, our claim is proved.
\item \textbf{Claim}: $$\iota_{K}(a) = a + g \muv$$ 
		OR	$$\iota_{K}(a) = a + g + z_{3} + z_{5} \muv$$
    \textit{Proof of claim:} We note that again, this Claim differs from \textbf{Claim} \ref{ClaimC} in the proof of Lemma \ref{lem:caseone} by the possible inclusion of $z_{3}$ and $z_{5}$ in the image. This is for morally the same reasoning as outlined in the proof of \ref{claimA3}, and we will deal with it in the same way. Showing that $a$ and $g$ must both appear in $\iota_{K}(a)$ is logically identical as in the proof of Lemma \ref{lem:caseone}. First, we show that $\iota_{K}(a) \neq 0$, which is precisely the same; then that $\iota_{K}(a) \neq a$ and in fact that $a$ and $g$ must appear in $\iota_{K}(a)$. In this case, as in the previous lemma, we can derive a contradiction using a linear algebra argument regarding what can appear in $\iota_{K}(d)$ and contribute a copy of $\cU g$, and what can appear in $\iota_{K}(c)$ and contribute a copy of $\cV g$. Despite the slightly different grading data in this case, the argument is morally the same; the if and only if statements mentioned in \ref{biglist} in particular follow the exact same argument and will not be replicated here. Once we have established that $a$ and $g$ both must appear in $\iota_{K}(a)$, it remains to consider what happens with $z_{3}$ and $z_{5}$; but again, using the Sarkar map and what we have shown in Claims \ref{claimA2} and \ref{claimA3} we can conclude that $z_{3}$ appears in $\iota_{K}(a)$ if and only if $z_{5}$ does.

\item It remains to again show that, should $g$ appear in the image of $\iota_{K}$ of any basis elements besides $g$ and $a$, it can be removed via a change of basis or homotoped away.

Indeed, we claim the only cases we must consider for Case Two are if $g$ appears in $\{\cU c, \cV d, x_{3}, x_{5}, z_{3}, z_{5} \}$; this follows from the fact that $g$ is not a boundary, and by argument similar to that used in \ref{b2b2}. 
    \item We can use the exact same skew-homotopy as in Case One to simultaneously eliminate $\cU g$ from the image of $\iota_{K}(c)$ and $\cV g$ from the image of $\iota_{K}(d)$; namely, the skew-homotopy defined by $H(b) = g$, $H(*) = 0$ for $*$ any other basis element.
    \item We have showed above in \ref{claimA2}  that $g$ appears in $\iota_{K}(z_{5}) \iff g$ appears in $\iota_{K}(z_{3})$. Then, the change of basis we require to remove $g$ from their images is:
        \begin{align*}
            z_{3} & \mapsto z_{3}^{'} = z_{3} + a \\
            z_{5} & \mapsto z_{5}^{'} = z_{5} + a \\
        \end{align*}
    \item We can mimic the same proof technique as used several times above to show that, using grading data and the Sarkar map, that $g$ appears in $\iota_{K}(x_{5}) \iff g$ appears in $\iota_{K}(x_{3})$. Then, the change of basis we require to remove $g$ from their images is:
        \begin{align*}
            x_{3} & \mapsto x_{3}^{'} = x_{3} + a \\
            x_{5} & \mapsto x_{5}^{'} = x_{5} + a \\
        \end{align*}
\end{enumerate}
\end{proof}
\section{The Mazur pattern of $4_{1}$ cannot have finite order in $\mathcal{C}$}
\label{sec:5final}

\subsection{Local equivalence}

We are now ready to prove \ref{prop:myknots}, which in turn immediately implies \ref{theorem:main}.
\begin{proof}
\label{locequiv}
Denote by $\mathcal{C}_{0}$ the almost-$\iota_{K}$-complex of the unknot, and let $\eta$ be the lone generator of $\mathcal{C}_{0}$. Note that $\iota_{K}$ acts on $\mathcal{C}_{0}$ in the obvious way, namely, $\iota_{K}(\eta) = \eta \muv$. Now, let $\mathcal{C}$ denote the almost-$\iota_{K}$-complex of the Mazur pattern of the figure eight knot, as described in Section \ref{sec:4invol}. We claim that there cannot exist an almost local equivalence $f$ such that 

$$f: \mathcal{C}_{0} \rightarrow \bigotimes_{i=1}^{N} \mathcal{C}$$

In particular, we will show that if an $f$ exists which satisfies $f\iota_{K} = \iota_{K}f \muv$ and which induces an isomorphism $f_{*}: H_{*}((\cU, \cV)^{-1}\mathcal{C}_{0}) \rightarrow H_{*}((\cU, \cV)^{-1}\bigotimes_{i=1}^{N} \mathcal{C})$, then such an $f$ fails to satisfy $f \omega \simeq \omega f$ where $\omega$ is the map $\omega = 1 + \iota_{K}$. This proof is a variation on that used in the paper \cite{HKPS_2020}, which was used to prove a similar result.

To begin, we note that $\eta$ is a cycle which generates the localized homology of $\mathcal{C}_{0}$, and thus $f(\eta)$ will generate the localized homology of $\bigotimes_{i=1}^{N} \mathcal{C}$. In particular, since $a \otimes a \otimes \cdots a \otimes a$ is a cycle, then

$$f(\eta) = a\otimes a \otimes \cdots a \otimes a + \lambda$$

\noindent where $\lambda$ is some linear combination of elements $\lambda_{i}$ which each satisfy that $\cU^{j} \lambda_{i}$ is a boundary for some $j >> 0$. For example, though $g$ is itself not a boundary of our chain complex, a priori $g$ may appear in $\lambda$. Next, we suppose for sake of contradiction that $f \omega \simeq \omega f$; specifically, we assume there exists some homotopy $H$ such that $f \omega = \omega f + H \partial + \partial H$. Then, for $\eta$ specifically, we have the equivalence

    \begin{equation} \label{eqn1}
                \begin{split}
                f \omega(\eta) & = \omega f(\eta) + H(\partial \eta) + \partial H(\eta) \\
                & = \omega f(\eta) \partial H(\eta)
                \end{split}
    \end{equation}
The left hand side becomes $f \omega(\eta) = f (1 + \iota_{K})(\eta) = f(\eta + \eta) = f(0) = 0$, which tells us that $\omega f(\eta)$ must be a sum of boundaries. Now, 

\begin{align*}
    \begin{split}
    \omega f (\eta) & = \omega (a\otimes a \otimes \cdots a \otimes a + \lambda) \\
    & = (1 + \iota_{K})(a\otimes a \otimes \cdots a \otimes a + \lambda) \\
    & = (a\otimes a \otimes \cdots a \otimes a + \lambda) + \iota_{K}(a\otimes a \otimes \cdots a \otimes a + \lambda) \\
    & = (a\otimes a \otimes \cdots a \otimes a + \lambda) + (a + g) \otimes (a + g) \otimes \cdots \otimes (a + g) + \iota_{K}(\lambda) \\
    & = (a\otimes a \otimes \cdots a \otimes a + \lambda) + g \otimes a \otimes \cdots \otimes a + \cdots + \iota_{K}(\lambda)
    \end{split}
\end{align*}

\noindent Then, there is at least one copy of $g \otimes a \otimes a \otimes \cdots \otimes a$ that appears in $\omega f(\eta)$; it remains to show that there is exactly one copy. In particular, we must show it is not canceled by anything in $\omega(\lambda)$. Indeed, $g \otimes a \otimes a \otimes \cdots \otimes a$ would only appear in $\lambda$ if $g$ appears in the image of some other basis element in a way that cannot be homotoped away; but as shown in \ref{lem:caseoneC} and \ref{lem:casetwoC}, this cannot happen. Thus, $0 = f\omega(\eta) \neq \omega f(\eta) + H(\partial \eta) + \partial H(\eta)$, and so $f \omega \not\simeq \omega f$, and $f$ fails to be a chain map.
\end{proof}

Upcoming work will address \ref{conj}; the proof that all iterates of $Q^{n}(4_{1})$ are infinite order is analogous to the argument herein; iterating $Q$ on the knot complex discussed in this paper lengthens each leg of the $L$-shaped complex by $1$, and generates other shapes that we have already dealt with. In particular, we have an odd number of acyclic complexes with the new $L$-shaped complex the unique summand of this shape. The proof that these iterates are linearly independent in $\mathcal{C}$ will require using a local equivalence argument to prove no two complexes are linear combinations of each other, an extension of the proof technique in \ref{locequiv}.

Recent work of \cite{Hedden_2021} tells us that the subgroup in $\mathcal{C}$ generated by the image of any pattern with winding number $w \neq 0$ will be infinite rank; upcoming work will reprove this in the case of $Q$, with an eye towards extending proof techniques to consider patterns with $w = 0$. 
\newpage

\bibliographystyle{amsalpha}
\bibliography{6paper.bib}

\end{document}